\documentclass[11pt]{article}
\usepackage{subfigure}
\usepackage{xcolor}
\usepackage{fullpage}
\usepackage{amsmath,amsthm,amssymb,amsfonts}
\usepackage{algorithm} 
\usepackage{algpseudocode}
\usepackage[hidelinks]{hyperref}
\hypersetup{
    colorlinks=true,
    linkcolor=blue,
    filecolor=magenta,      
    urlcolor=cyan,
    pdftitle={Overleaf Example},
    pdfpagemode=FullScreen,
}

\usepackage[noblocks]{authblk}
\usepackage{tabularx}
\usepackage{booktabs}

\usepackage{mathtools}
\usepackage{tcolorbox}
\usepackage{bbm}
\usepackage{todonotes}
\usepackage{tensor}
\usepackage[font=small,labelfont=bf]{caption}
\usepackage{subcaption}
\usepackage{graphicx}
\usepackage{sectsty}

\usepackage{adjustbox} 
\usepackage{enumitem} 

\newcommand{\Trace}{\mathop{\bf tr}}
\newcommand{\tr}{\mathsf{ T}}
\newcommand{\RR}{\mathbb{R}}

\newtheorem{theorem}{Theorem}

\newtheorem{remark}{Remark}[section]
\newtheorem{lemma}{Lemma}

\newtheorem{corollary}{Corollary}[section]

\newtheorem{assumption}{Assumption}[section]

\newcommand{\bsf}[1]{\textsf{\LARGE\textbf{#1}}}
\newcommand{\bigline}{\\ \vrule height2pt width 5 in depth 0pt\newline\noindent}

\newcommand{\pstar}{p^\star}

\newcommand{\gstar}{g^\star}
\newcommand{\Xstar}{X^\star}
\newcommand{\ystar}{y^\star}

\newcommand{\Amap}{\mathcal{A}}
\newcommand{\Ajmap}{\Amap^*}
\newcommand{\Dist}{\mathrm{dist}}

\DeclareMathOperator*{\argmin}{\arg\!\min}

\newcommand{\innerproduct}[2]{\left \langle #1, #2 \right\rangle }

\definecolor{moccasin}{rgb}{0.98, 0.92, 0.84}
\newtcolorbox{mybox}{colback=moccasin,
colframe=moccasin}

\usepackage[nameinlink]{cleveref}
\crefname{equation}{}{}
\crefname{theorem}{Theorem}{Theorems}
\crefname{corollary}{Corollary}{Corollaries}
\crefname{example}{Example}{Examples}
\crefname{assumption}{Assumption}{Assumptions}
\crefname{lemma}{Lemma}{Lemmas}
\crefname{proposition}{Proposition}{Propositions}
\crefname{figure}{Figure}{Figures}
\crefname{table}{Table}{Tables}
\crefname{section}{Section}{Sections}
\crefname{appendix}{Appendix}{Appendices}
\Crefname{equation}{}{}
\Crefname{theorem}{Theorem}{Theorems}
\Crefname{corollary}{Corollary}{Corollaries}
\Crefname{example}{Example}{Examples}
\Crefname{lemma}{Lemma}{Lemma}
\Crefname{proposition}{Proposition}{Propositions}
\Crefname{figure}{Figure}{Figures}
\Crefname{table}{Table}{Tables}
\Crefname{section}{Section}{Sections}
\Crefname{appendix}{Appendix}{Appendices}

\newcommand{\sn}{\mathbb{S}^n}
\newcommand{\snp}{\mathbb{S}^n_+}
\newcommand{\dsolution}{\Omega_{\mathrm{D}}}

\newcommand{\bigO}{\mathcal{O}}
\newcommand{\xstar}{x^\star}

\newcommand{\calL}{\mathcal{L}}
\newcommand{\calS}{\mathcal{S}}
\newcommand{\SBM}{\textsf{SBM}}

\newcommand{\alg}{\texttt{BALA}}
\newcommand{\CGAL}{\texttt{CGAL}}

\usepackage[numbers,merge,sort&compress]{natbib}
\begin{document}

\title{
 \bf \Large A Bundle-based Augmented Lagrangian Framework: Algorithm,  Convergence, and Primal-dual Principles
\thanks{This work is supported by NSF ECCS 2154650, NSF CMMI 2320697, and NSF CAREER 2340713. Emails: fliao@ucsd.edu; zhengy@ucsd.edu. } 
}
\author[1]{Feng-Yi Liao}
\author[1]{Yang Zheng}
\affil[1]{\small Department of Electrical and Computer Engineering, University of California San Diego}
\date{\small \today \vspace{-5ex}} 

\maketitle

\begin{abstract}
    We propose a new bundle-based augmented Lagrangian framework for solving constrained convex problems. Unlike the classical (inexact) augmented Lagrangian method (ALM) that~has a nested double-loop structure, our framework features a \textit{single-loop} process. Motivated by the proximal bundle method (PBM), we use a \textit{bundle} of past iterates to approximate the subproblem in ALM to get a computationally efficient update at each iteration. We establish sub-linear convergences for primal feasibility, primal cost values, and dual iterates under mild assumptions. With further regularity conditions, such as quadratic growth, our algorithm enjoys \textit{linear} convergences. Importantly, this linear convergence can happen for a class of conic optimization problems, including semidefinite programs. Our proof techniques leverage deep connections with inexact ALM and primal-dual principles with PBM.   
\end{abstract}

\section{Introduction}
We consider the following constrained convex optimization problem with decision variable $x \in \mathbb{R}^n$:  
\begin{equation}
    \label{eq:primal}
    \begin{aligned}
        \pstar := \min_{x \in \RR^n} & \quad \innerproduct{c}{x}\\ \mathrm{subject~to} & \quad \Amap x = b, \\
        & \quad  x \in \Omega,
    \end{aligned}
    \tag{P}
\end{equation}
where $\langle \cdot, \cdot \rangle$ denotes the standard inner product in $\mathbb{R}^n$ and the problem data includes a linear map $\Amap:\RR^n \to \RR^m$, two constant vectors $b \in \mathbb{R}^m$ and $c \in \mathbb{R}^n$, and a compact convex set $\Omega \subseteq \RR^n$. We call \Cref{eq:primal} as the \textit{primal} formulation. Any optimization problem can be associated with a \textit{Lagrange dual} formulation. The dual of \Cref{eq:primal} plays an important role in our algorithmic design, and the dual problem \cref{eq:dual}~will be introduced in \Cref{subsection:bundle}. 
{We assume the slater's condition holds for \cref{eq:primal}, and thus its dual is solvable and strong duality holds \cite[Section 5.2.3]{boyd2004convex}.}

\vspace{3pt}
\noindent \textbf{Applications.} Problem \eqref{eq:primal} covers almost all conic programs which is a very general class of convex optimization problems \cite{wolkowicz2012handbook}. 
For example, a standard linear program (LP) is in the form of \eqref{eq:primal} with $\Omega = \{x \in \RR^n : x_i \geq 0, i=1,\ldots,n\}$. If an optimal solution $\xstar$  satisfies $\|\xstar\| \leq \alpha$, we can replace $\Omega$ by the compact set $\{x\in \RR^n : x \geq0, \|x\| \leq \alpha \}$. Adding a redundant constraint to guarantee compactness is common in designing algorithms for solving conic programs \cite{hough2024primal,liao2023overview,ding2023revisiting,yurtsever2021scalable}. 
Problem \eqref{eq:primal} has found increasing applications in signal processing, robotics,  machine learning, and control theory  \cite{wolkowicz2012handbook,zheng2021chordal}. Classical applications include combinatorial problems such as max-cut \cite{goemans1995improved}, control problems such as searching for a Lyapunov function \cite{papachristodoulou2002construction}, graph problems such as community detection \cite{abbe2014decoding}, data science problems such as phase-retrieval \cite{candes2013phaselift} and low-rank matrix recovery \cite{chen2013low}. More recently, we have seen emerging applications in automated algorithm analysis 
 such as performance estimation  \cite{goujaud2024pepit}, and trustworthy machine learning such as Lipschitz constant estimation~\cite{wang2024scalability}, and many others.

\vspace{3pt}
\noindent \textbf{Augmented Lagrangian method}. Many algorithms have been developed to solve \cref{eq:primal} \cite{wolkowicz2012handbook,zheng2021chordal}.~In this paper, we focus on the augmented Lagrangian method (ALM), which is a conceptually simple yet algorithmically powerful framework for constrained optimization \cite{rockafellar1976augmented,hestenes1969multiplier,powell1969method}. For iterations $k = 1,2,\ldots,$ the ALM has two steps
\begin{subequations}
    \label{eq:alm-steps}
    \begin{align}
        x_{k+1} & \in  \argmin_{x \in \Omega}\;
        \calL_{\rho}(x,y_k),
        \label{eq:ALM-exact-min}\\
            y_{k+1} & = y_k + \rho(b-\Amap x_{k+1}) \label{eq:ALM-exact-update},
    \end{align}
\end{subequations}
where $\calL_\rho$ is the augmented Lagrangian function defined as 
\begin{equation}
\begin{aligned}
\label{eq:AL}
        \calL_\rho(x,y) = & \innerproduct{c}{x} + \innerproduct{y}{b-\Amap x} + \frac{\rho}{2}\|b - \Amap x \|^2,  
\end{aligned}    
\end{equation}
with a fixed constant $\rho >0$. The ALM is a primal and dual framework: step \eqref{eq:ALM-exact-min} aims to get a primal candidate, and \eqref{eq:ALM-exact-update} is a dual ascent update. It is also closely related to the proximal point method (PPM) \cite{rockafellar1976monotone}.
The iterations in \cref{eq:alm-steps} are conceptually very simple and enjoy fast and stable convergence behaviors for convex optimization. 

However, one main computational difficulty for the ALM lies in finding an exact solution of~\eqref{eq:ALM-exact-min}. As already suggested by \cite{rockafellar1976augmented}, we often consider an inexact version of ALM (iALM), which replaces \eqref{eq:ALM-exact-min} by \vspace{-1mm}
\begin{equation}
\label{eq:ALM-inexact-min}
    x_{k+1} \approx  \argmin_{x \in \Omega}\; \calL_{\rho}(x,y_k) 
\end{equation}
with $x_{k+1}$ being an inexact solution. It is shown by many subsequent studies  \cite{luque1984asymptotic,cui2019r,xu2021iteration,liao2024inexact} that the iALM has good convergence properties when controlling the inexactness properly. Some recent solvers have been developed under the iALM framework \cite{chen2016semismooth,sun2020sdpnal+,yurtsever2019conditional}. Despite the great advances, the iALM does not specify how to solve \eqref{eq:ALM-inexact-min}. An efficient strategy for \cref{eq:ALM-inexact-min} is often problem-dependent and requires special care.

\vspace{3pt}
\noindent \textbf{Related works}.  One main approach to acquiring an acceptable $x_{k+1}$ is to develop another iterative algorithm as a subroutine, such as the accelerated proximal gradient methods \cite{beck2009fast} or Frank-Wolf methods \cite{frank1956algorithm}. The semi-smooth Newton method in \cite{yang2015sdpnal+,zhao2010newton} and Nesterov's optimal first-order method in \cite{lan2016iteration} are another two good choices. 
This approach is useful when one can exploit the problem-specific properties to get a solution to  \eqref{eq:ALM-inexact-min} efficiently. However, a general solution strategy is nontrivial. 
Another approach is to design a sequence of simpler sets $\{\Omega_{k}\} $ to approximate $\Omega$ and solve the sequence of approximated problems until an acceptable solution is acquired. For instance, when set $\Omega$ involves a positive semidefinite constraint such as $\Omega = \snp \cap \mathcal{M}$ where $\mathcal{M} \subseteq \sn$ is a simple convex set, there are different approximation strategies \cite{ahmadi2017sum,zheng2022block,liao2022iterative}. A simple approach is to consider the set of diagonal dominant  matrices ($\mathrm{DD}_n :=\{X \in \sn : X_{ii} \geq \sum_{j \neq i} |X_{ij}|, \forall i = 1,\ldots,n\}$) or 
scaled diagonal dominant (SDD) matrices ($\mathrm{SDD}_n:=\{X \in \sn : \exists d \in \RR^n_{++}, \mathrm{diag}(d)X\mathrm{diag}(d) \in \mathrm{DD}_n \}$) and replace $\Omega$ by $\mathrm{DD}_n \cap \mathcal{M}$ or $\mathrm{SDD}_n \cap \mathcal{M}$. 
This leads to simpler linear programs or second-order cone programs \cite{ahmadi2017sum}. However, all these approximation strategies are largely heuristics, and no convergence guarantees are known \cite{zheng2021chordal}. 

A third approach, known as the Burer-Monteiro (BM) factorization \cite{burer2003nonlinear}, is to consider an approximation set $\Omega_r = \{V V^\tr: V \in \RR^{n \times r} \} \cap \mathcal{M}$ with $r \leq n$. Clearly, we have $\Omega_r \subseteq \Omega$ as $\{V V^\tr: V \in \RR^{n \times r} \} \subseteq \snp$, and increasing the value of $r$ improves the approximation quality. With this observation, the authors in \cite{monteiro2024low} suggest solving the restricted problem, starting with $r = 1$ and gradually increasing the value of $r$ by checking certain eigenvalue information. BM factorization-based approaches enjoy lower complexity as the dimension of $\RR^{n \times r}$ is usually much less than that of $\sn$. However, the subproblem becomes nonconvex as the variable becomes  $V \!\in\! \RR^{n \times r}$ and the cost function becomes nonconvex.  It is generally difficult to quantify the cost value gap $\calL_\rho(w_{k+1},y_k) - \min_{x\in \Omega}\calL_{\rho}(x,y_k)$. The global convergence of BM factorization-based approaches requires special treatments and is still under active investigation \cite{tang2024feasible,wang2023decomposition}.

\vspace{3pt}

\noindent \textbf{Our contributions.}  
In this paper, we move beyond the iALM and introduce a new \textit{bundle-based} augmented Lagrangian framework. 
Our key idea is to leverage \textit{a bundle of past iterates} to approximate the convex set $\Omega_k \subset \Omega$ in a principled way, which enables an exact and efficient solution to an approximated subproblem. Unlike the heuristic matrix approximations in \cite{ahmadi2017sum,zheng2022block,liao2022iterative} and the low-rank BM factorization \cite{burer2003nonlinear}, our inner approximation $\Omega_k \subset \Omega$ is guided by both \textit{current}~dual information and \textit{aggregate} primal progress. Indeed, our set $\Omega_k$ can be as simple as a \textit{line segment}~so~that each approximated subproblem admits an analytical solution with no subroutine required. Thus, our~proposed \underline{B}undle-based \underline{A}ugmented \underline{L}agrangian \underline{A}lgorithm (\texttt{BALA}) features a single-loop process, and all iterations only require simple computation. 

Under mild assumptions, we establish asymptotic and sublinear convergences for primal feasibility, primal cost values, and dual iterates (\Cref{theorem:asymptotic,theorem:convergence-bundle-dual,thm:average-iterate}). Moreover, with further regularity conditions on \textit{quadratic growth} in the dual and \textit{quadratic closeness} in the dual approximation, we establish fast linear convergences of \alg{} in \Cref{thm:linear-convergence}.  Importantly, this fast linear convergence can happen for a class of conic programs, building on recent advances \cite{ding2023revisiting,liao2024inexact}.  
Our proof techniques rely on a deep primal and dual relationship between \alg{} with the proximal bundle method \cite{diaz2023optimal}, which is analogous to the connection between the exact ALM and PPM revealed in \cite{rockafellar1976augmented}. 

Finally, our algorithm \alg{} shows superior performances in both theory and experiments compared to recent developments on the \textit{conditional-gradient} augmented Lagrangian framework~\cite{yurtsever2019conditional,yurtsever2021scalable,garber2023faster}. In theory, the dual iterates in these developments have no convergence guarantees, as the algorithms in  \cite{yurtsever2019conditional,yurtsever2021scalable,garber2023faster} behave more like a penalty method. In experiments, our algorithm \alg{} observes linear convergence for a range of conic optimization problems, while the algorithms in \cite{yurtsever2019conditional,yurtsever2021scalable,garber2023faster} only have sublinear convergences.

\vspace{3pt}
\noindent \textbf{Paper organization.} 
This paper is organized as follows. \Cref{section-preliminary} reviews the inexact ALM~and the proximal bundle method. \Cref{section:BALA} introduces our proposed algorithm \alg{}. \Cref{section:convergence} establishes the convergence guarantees of \alg{}, and numerical experiments are shown in \Cref{section:experiment}. We conclude the paper in \Cref{section:conclusion}. Further discussions, proofs, and extra experiments are provided in the~appendix.

\section{Preliminaries}
\label{section-preliminary}
Our algorithmic design is based on the Augmented Lagrangian framework \cite{rockafellar1976augmented} and is also closely related to the proximal point method \cite{rockafellar1976monotone}. 
This section reviews these two foundational frameworks. 

\subsection{Augmented Lagrangian and its inexact version}
Fix a constant $\rho > 0$. The augmented Lagrangian function $\calL_\rho(x,y)$ for \cref{eq:primal} is defined in \cref{eq:AL},~where $y \in \mathbb{R}^m$ is a dual variable associated with the equality constraint. Starting with an initial dual variable $y_1 \in \mathbb{R}^m$, the augmented Lagrangian method (ALM) generates a sequence of primal and dual iterates $(x_k,y_k), k = 1, 2, \ldots, $ via \cref{eq:ALM-exact-min,eq:ALM-exact-update}. Very often, the primal update \cref{eq:ALM-exact-min} is computationally challenging due to the (potentially complicated) convex constraint $x \in \Omega$.   

In the literature, a classical and widely studied strategy is to solve \cref{eq:ALM-exact-min} inexactly, i.e., replacing it with an inexact update \cref{eq:ALM-inexact-min}. This strategy leads to a family of inexact ALM algorithms, which traces back to the seminal work by Rockafellar \cite{rockafellar1976augmented} and was widely discussed in many later works, e.g., \cite{luque1984asymptotic,liu2019nonergodic,xu2021iteration}. It is clear that the convergence of inexact ALM depends on proper inexactness control in \cref{eq:ALM-inexact-min}. As suggested in \cite{rockafellar1976augmented}, the next point $x_{k+1}$ from \cref{eq:ALM-inexact-min} should be an $\epsilon_k$-suboptimal solution to \cref{eq:ALM-exact-min}, i.e., $x_{k+1} \in \Omega$ and 
\begin{equation}
    \label{eq:subproblem-error}
    \calL_{\rho}(x_{k+1},y_k) - \min_{x \in \Omega}\; \calL_{\rho}(x,y_k) \leq \epsilon_k.
\end{equation}
We summarize this inexact ALM procedure in \Cref{alg:alm}. 
Rockafeller \cite{rockafellar1976augmented} has established that \Cref{alg:alm} finds an optimal solution as long as  $\sum_{k=1}^{\infty} \epsilon_k < \infty$.
\begin{theorem}[{\cite[Theorem 4]{rockafellar1976augmented}}]
    \label{thm:ALM-convergence}
   Suppose the Slater's condition holds for \cref{eq:primal}. Let $\{x_k,y_k\}$ be a sequence from \Cref{alg:alm} with $\sum_{k=1}^\infty \epsilon_k < \infty$. Then, the sequence $\{y_k\}$ converges asymptotically to an optimal dual solution. Moreover, 
   for all $k \geq 1$, we have  
   \begin{subequations}
   \label{eq:ALM-residual}
       \begin{align}
           \|\Amap x_{k+1} - b\|&  =  \|y_k - y_{k+1}\|/\rho, \label{eq:ALM-affine}
            \\
            \innerproduct{c}{x_{k+1}} - \pstar & \leq \epsilon_k             +   ( \|y_k\|^2 -\|y_{k+1}\|^2)/(2\rho), \label{eq:ALM-cost}
       \end{align}
   \end{subequations}
     which confirms the convergence of the primal affine feasibility and the cost value gap, i.e., we have $\lim_{k\to \infty} \|\Amap x_{k+1} - b\| = 0$, and $\lim_{k\to \infty}  \innerproduct{c}{x_{k+1}} -p^\star = 0$. 
\end{theorem}

\begin{algorithm}[t]
\caption{Inexact Aug. Lagrangian Method (iALM)}\label{alg:alm}
\begin{algorithmic}[1]
\Require $y_1 \in \RR^m, \rho > 0, \{\epsilon_k\} , k_{\max} \in \mathbb{N}$
\For{$k=1, 2, \ldots,k_{\max} $}
    \State Solve \cref{eq:ALM-inexact-min} for an inexact primal $x_{k+1}$, satisfying~\cref{eq:subproblem-error}.
    \State Update the dual variable $y_{k+1}$ via \cref{eq:ALM-exact-update}. 
\EndFor
\end{algorithmic}
\end{algorithm}

The Slater's condition ensures a dual optimal solution exists. We note that the convergence of the primal feasibility $\|\Amap x_{k+1} - b\|$ and the cost value gap $\innerproduct{c}{x_{k+1}} - \pstar $ is a direct consequence of that of $y_k$ thanks to \Cref{eq:ALM-residual}. However, the convergence of the primal sequence $\{x_k\}$ may not be guaranteed in \Cref{thm:ALM-convergence}. This~theorem does not provide the convergence rate either. Under some mild assumptions on the dual function of \Cref{eq:primal} (such as quadratic growth), the iALM in \Cref{alg:alm} is known to enjoy linear convergence in terms of dual iterates \cite{rockafellar1976augmented,luque1984asymptotic}, Karush–Kuhn–Tucker residuals \cite{cui2019r}, and primal iterates \cite{liao2024inexact}. Thanks to its stable and fast convergence performances, iALM has been widely used to develop efficient solvers for constrained convex optimization; see, e.g.,  \cite{birgin2014practical,yang2015sdpnal+,yurtsever2021scalable}. 

\subsection{Proximal point method and its bundle version} 
\label{subsection:bundle}

It is known that running the ALM for the primal \cref{eq:primal} is equivalent to executing a proximal point method (PPM) for its dual \cite{rockafellar1976augmented}. One key step in~\Cref{thm:ALM-convergence} is to establish the convergence of the dual iterates $y_k$, which is guaranteed by the analysis of PPM \cite{rockafellar1976monotone}. We~here~review the dual of \cref{eq:primal}, PPM and its bundle version. 

The Lagrangian of \eqref{eq:primal} is  $\calL(x,y) =\innerproduct{c}{x} + \innerproduct{y}{b-\Amap x}$. The dual function $\theta$ is thus given by
\begin{equation} \label{eq:dual-function}
     \theta(y) = \min_{x\in \Omega} \calL(x,y)  = \innerproduct{b}{y} + \min_{x \in \Omega } \innerproduct{c-\Ajmap y}{x}, 
\end{equation}
where $\Ajmap$ denotes the adjoint map of $\Amap$, i.e., $\langle \Amap x, y \rangle = \langle x, \Ajmap y  \rangle, \forall x \in \mathbb{R}^n, y \in \mathbb{R}^m$. The Lagrange dual of \eqref{eq:primal} is  
\begin{equation}
    \label{eq:dual}
    \begin{aligned}
       \theta^\star := \max_{y \in \mathbb{R}^m}  \; \theta(y). 
    \end{aligned}
    \tag{D}
\end{equation}
If \eqref{eq:ALM-exact-min} is solved exactly, then $y_{k+1}$ in \eqref{eq:ALM-exact-update} is the same as a proximal step on the dual function $g$ centered at $y_k$ \cite[Proposition 6]{rockafellar1976augmented}. 

We next review the PPM framework. Following the conventions of optimization, we consider a minimization problem 
\begin{equation} \label{eq:PPM-function}
f^\star := \min_{y\in \RR^n} f(y)
\end{equation}
where $f:\RR^n \to \RR$ is a convex (potentially nonsmooth) function. Problem \cref{eq:dual} is in the form of \cref{eq:PPM-function} by viewing $f=-g$. 
We define the \textit{proximal mapping} as
\begin{equation} \label{eq:PPM-subproblem}
    \text{prox}_{\alpha f}(y_k):=\argmin_{y \in \mathbb{R}^n}\; f(y) + \frac{1}{2\alpha} \left\|y - y_k\right\|^2,
\end{equation}
where $\alpha > 0$ is a given constant.  
Starting with any initial point $y_1$, the PPM generates a sequence of points as follows
\begin{equation}
    \label{eq:PPM}
     y_{k+1} =   \mathrm{prox}_{\alpha f}(y_k), \quad k = 1, 2, \ldots.
\end{equation}
Thanks to~the~quadratic term, the strong convexity of the cost in \cref{eq:PPM-subproblem} ensures a unique solution, and thus the iterates \cref{eq:PPM} are well-defined. 
The convergence of the (exact) PPM \cref{eq:PPM} is very well-studied \cite{rockafellar1976monotone,luque1984asymptotic,leventhal2009metric}. However, the proximal mapping \cref{eq:PPM-subproblem} is often computationally hard to evaluate (unless for special functions; see \cite[Chapter 6]{beck2017first} for a list of such cases). 

A classical and widely studied approach to approximate and simplify the original function $f$ is to use a \textit{bundle} of its (sub)gradients; see \cite{lemarechal1994condensed} for an overview. This leads to a family of proximal bundle methods (PBMs). 
Instead of \cref{eq:PPM}, the PBM updates a candidate point as $z_{k+1} = \mathrm{prox}_{\alpha f_k}(y_k)$, where the original $f$ is replaced by an approximation $f_k$ (conditions are given below). To test the quality of $z_{k+1}$, the PBM employs a criterion
\begin{equation}
    \label{eq:testing-bundle}
    f(y_{k}) - f(z_{k+1}) \geq \beta(f(y_{k}) - f_{k}(z_{k+1})), 
\end{equation}
where $\beta \in (0,1)$ is a predefined constant. If \eqref{eq:testing-bundle} holds, then we set $ y_{k+1} = z_{k+1}$ (this is called a \emph{descent step}). Otherwise, we set $y_{k+1} = y_k$ (known as a \textit{null} step). Regardless of a descent or null step, the PBM improves the next approximation $f_{k+1}$ using the new information $z_{k+1}$. 

\begin{assumption}
    \label{assump:bm} 
    The function $f_{k+1}$ satisfies three conditions: 
    
    \vspace{-2mm}
    \begin{enumerate}[leftmargin=*]
    \setlength{\itemsep}{0pt}
        \item \textbf{Lower approximation:} $f_{k+1}(y) \leq f(y), \forall y \in \mathbb{R}^m.$ \vspace{-1pt}
        \item \textbf{Subgradient:} There exists $ v_{k+1} \in \partial f(z_{k+1})$ such that 
        $$f_{k+1}(y) \geq f(z_{k+1}) + \innerproduct{v_{k+1}}{y-z_{k+1}}, \forall y \in \RR^m.$$ 
        
        \vspace{-8pt}
        
        \item \textbf{Aggregation:} For a null step, we require
        $$f_{k+1}(y) \geq f_k(z_{k+1}) + \innerproduct{s_{k+1}}{y-z_{k+1}},\forall y \in \RR^m,$$
        where $s_{k+1} = \frac{1}{\alpha}(y_k - z_{k+1}) \in \partial f_{k}(z_{k+1})$.
    \end{enumerate}
\end{assumption}
\vspace{-6pt}

\begin{algorithm}[t]
\caption{Proximal Bundle method (PBM)}\label{alg:bundle}
\begin{algorithmic}[1]
\Require $y_0 \in \RR^m, \alpha > 0, \beta \in (0,1),  k_{\max}\in \mathbb{N}$
\For{$k=0,1,2, \ldots,k_{\max}$}
    \State Compute the candidate point $z_{k+1} = \mathrm{prox}_{\alpha f_k}(y_k)$.
    \If{\cref{eq:testing-bundle} is satisfied}
        \State Set $y_{k+1} = z_{k+1}$, \hfill \Comment{\textit{descent step}}
    \Else
        \State Set $y_{k+1} = y_k$.  \hfill \Comment{\textit{null step}}
    \EndIf
    \State Update the model $f_{k+1}$ following \cref{assump:bm}.
\EndFor
\end{algorithmic}
\end{algorithm}

Requirement (1) means that $f_{k+1}$ is a global lower approximation of $f$, and requirement (2) means that $f_{k+1}$ should include the subgradient lowerbound of $f$ at $z_{k+1}$. The last requirement indicates that $f_{k+1}$ should aggregate the subgradient of $f_{k}$ at $z_{k+1}$. We list the PBM procedure in \Cref{alg:bundle}. Its convergence results are summarized below.

\begin{theorem}[{\cite[Theorem 2.1]{diaz2023optimal}}]
\label{thm:convergence-bundle}
    Consider \cref{eq:PPM-function} where $f$ is an $M$-Lipschitz convex function. Suppose its solution set $ S\!=\! \argmin_{y } f(y)$ is non-empty. Fix $\alpha > 0$ and $\beta \in (0,1)$ in \Cref{alg:bundle}. Given an $\epsilon > 0$, it holds that
\vspace{-12pt}
\begin{enumerate}[leftmargin=*]
\setlength{\itemsep}{0pt}
    \item A descent step will happen after at most $\bigO(\epsilon^{-2})$ number of consecutive null steps;
     \item The total number of steps to find an iterate $y_k$ satisfying $f(y_k) - f^\star \leq \epsilon$ is $\bigO\left(M^2\epsilon^{-3} \right)$. Moreover, if $\alpha = 1/\epsilon$, then the total number of steps reduces to $\bigO\left(M^2\epsilon^{-2} \right)$.
\end{enumerate}  
\end{theorem}

\vspace{-2pt}

If $f$ satisfies further regularity conditions such as \textit{quadratic growth}, the convergence rate of \Cref{alg:bundle} can be improved; see \cite[Table 1]{diaz2023optimal}. Problem \cref{eq:PPM-function} is unconstrained, and extensions to the constrained case can be found in \cite[Lemma 2.2]{liao2023overview}.

\subsection{Our motivation of \texttt{BALA}}
\label{subsection:motivation}
While the iALM in \Cref{alg:alm} provides stable and fast convergences in practice, it does not inform how to solve the subproblem \cref{eq:ALM-inexact-min} to satisfy the condition \cref{eq:subproblem-error}. In the literature, different subroutines are used to achieve the inexactness, e.g., the semi-smooth Newton method in \cite{yang2015sdpnal+,zhao2010newton} and Nesterov's optimal first-order method in \cite{lan2016iteration}. Almost all existing iALM-based algorithms have nested loops, including a customized algorithm as the inner loop to solve~\cref{eq:ALM-inexact-min}.   
On the other hand, the PBM in \Cref{alg:bundle} can be viewed as one single loop algorithm, in which a unique testing criterion \cref{eq:testing-bundle} is used to determine whether or not to move the iterate.   

Building on the primal and dual insights between ALM and PPM in \cite{rockafellar1976augmented}, we develop a new \textit{bundle-based} augmented Lagrangian framework. This framework leverages \textit{a bundle of past iterates} to approximate the convex set $\Omega_k \subset \Omega$, which enables an exact and efficient solution to an approximated subproblem. This approach aligns in spirit with the PBM that approximates the original complicated function $f$ using a bundle of its past \textit{subgradients}. We will present this framework and detail the $\Omega_k \subset \Omega$ construction in \cref{section:BALA}. Our proposed \underline{B}undle-based \underline{A}ugmented \underline{L}agrangian \underline{A}lgorithm (\texttt{BALA}) features a single-loop structure similar to that of \Cref{alg:bundle}. Indeed, our proposed \texttt{BALA} shares a deep primal and dual relationship with the PBM in \Cref{alg:bundle}, analogous to the connection between the exact ALM and PPM revealed in \cite{rockafellar1976augmented}. We establish non-asymptotic convergences of \texttt{BALA}, which show superior performances compared to the conditional-gradient augmented Lagrangian framework \cite{yurtsever2019conditional,yurtsever2021scalable,garber2023faster} in both theory and numerical experiments.

\section{A Bundle-based Augmented Lagrangian Framework and \alg{}} \label{section:BALA}

This section presents a new bundle-based augmented Lagrangian framework and a specific algorithm \alg. Similar to (inexact) ALM, our approach is essentially based on a primal and dual framework, utilizing both primal and dual information to update iterates. We recall the primal problem is shown in \cref{eq:primal}, and for notational convenience, we rewrite its dual problem \cref{eq:dual} in the form of \cref{eq:PPM-function} as 
\begin{equation}
    \label{eq:dual-minimizatio}
    \begin{aligned}
        g^\star:= \min_{y \in \mathbb{R}^m}  \; g(y),
    \end{aligned} \vspace{-1mm}
\end{equation}
where we denote $g(y) := -\theta(y), \forall y \in \mathbb{R}^m$. Note that the optimal value is related as $g^\star = - \theta^\star$. 

\subsection{A conceptual idea of inner approximations}
The primal update \cref{eq:ALM-exact-min} is the computationally expensive step in ALM, often due to the potentially complicated convex set $\Omega$. For example, $\Omega$ often involves the set of positive semidefinite (PSD) matrices in many applications \cite{goemans1995improved,lanckriet2004learning}. Even testing the membership of $x \in \Omega$ typically requires a complexity of order $\mathcal{O}(\tilde{n}^3)$, where $\tilde{n}$ denotes the PSD matrix dimension. 

Unlike the classical choice of getting an inexact solution \cref{eq:subproblem-error}, we aim to approximate the set $\Omega$ by another simpler convex set $\Omega_k \subset \Omega$, such that the following approximated subproblem admits an efficient and exact solution
\begin{equation}
\label{eq:ALM-exact-set-min-main-text}
    w_{k+1} \in \argmin_{x \in \Omega_{k}}\; \calL_{\rho}(x,y_k).
\end{equation}
For example, one may use the set of diagonally dominant matrices to approximate PSD matrices, leading to a set of linear constraints. This idea has indeed been used widely in other contexts (but not in the ALM framework); see e.g., \cite{ahmadi2017sum,zheng2022block,liao2022iterative}. When applying this approximation idea in the ALM, if we directly replace the primal update \cref{eq:ALM-exact-min} with \cref{eq:ALM-exact-set-min-main-text}, then the algorithm convergence would heavily depend on the approximation quality $\Omega_k \subset \Omega$. An $\epsilon$-optimal solution would mean to have a very accurate inner approximation, which poses two challenges 1) constructing an accurate approximation $\Omega_k \subset \Omega$ is itself a non-trivial problem \cite{song2023approximations} and 2) solving \cref{eq:ALM-exact-set-min-main-text} exactly with this $\Omega_k$ might not be computationally simple \cite{zheng2022block}. 

To avoid the challenges, our key idea is to exploit the past iterates to iteratively update the approximation $\Omega_k$, which can guide our search process and reach a global minimizer. This is analogous to the PBM that iteratively approximates the original (potentially complicated) convex function $f$. 

\subsection{A novel bundled-based inner approximation}
\label{subsection:bundle-inner-approximation}

Our starting point is to relax the requirement of obtaining an $\epsilon$-optimal solution to \cref{eq:ALM-exact-min} from \cref{eq:ALM-exact-set-min-main-text}. Instead, we will use $w_{k+1}$ to update the next inner approximation $\Omega_{k+1} \subset \Omega$. 
It is important to note -- and we will elaborate on this later -- that these approximations are typically not nested, i.e.,  $\Omega_k \not\subseteq \Omega_{k+1}$. At each iteration, the inner approximation $\Omega_{k}$ can be chosen to be very simple, enabling an exact and computationally efficient solution to \cref{eq:ALM-exact-set-min-main-text}. Indeed, each $\Omega_{k}$ 
could be as simple as a \textit{line segment}, such that \cref{eq:ALM-exact-set-min-main-text} admits an analytical solution (no other subroutines are required). 

To achieve this, we will use a dual function to guide the update of $\Omega_{k+1}$ and employ a testing step similar to the PBM. Each iteration of our bundled-based augmented Lagrangian framework has three steps: a) computing a pair of primal and dual candidates from \cref{eq:ALM-exact-set-min-main-text}; b) testing the descent progress; c) updating the approximation $\Omega_{k+1}$. 

\vspace{3pt}

\noindent \textbf{Step a) Computing a pair of candidates.} 
We compute an exact optimal solution $w_{k+1}$ to \cref{eq:ALM-exact-set-min-main-text} as our primal candidate. Any convex inner approximation $\Omega_1 \subset \Omega$ and dual initial point $y_1 \in \mathbb{R}^m$ are acceptable for the first iteration. We perform a dual ascent on $\calL_{\rho}(w_{k+1},\cdot)$ to get a dual candidate:  
\begin{equation} \label{eq:dual-update-ALM}
    z_{k+1} =  y_k + \rho ( b- \Amap w_{k+1}).
\end{equation}

\vspace{3pt}

\noindent
\textbf{Step b) Testing the descent progress.} We then test if the candidate $(w_{k+1},z_{k+1})$ makes sufficient descent progress. In particular, we examine the cost drop of the dual function, i.e.,  
$g(y_{k}) - g(z_{k+1})$. Ideally, we want the dual cost to decrease sufficiently. 
Since we have approximated the primal augmented Lagrangian in \cref{eq:ALM-exact-set-min-main-text}, it is useful to define an induced \textit{approximated} dual function \vspace{-1mm}
\begin{equation} 
\label{eq:approximated-dual}
g_k(y) = - \min_{x\in\Omega_k} \calL(x,y).  \vspace{-1mm}
\end{equation}
Instead of fixing a constant threshold, we use a dynamic test
\begin{equation}
\label{eq:test}
    g(y_{k}) - g(z_{k+1}) \geq \beta (g(y_{k}) - g_k(z_{k+1})),
\end{equation}
where $\beta\! \in\! (0,1)$. The term $g(y_{k}) - g_k(z_{k+1})$  can be viewed as an approximated dual cost drop. By definition, we have \begin{equation}
    \label{eq:g-gk}
    g(z_{k+1}) \!=\! -\min_{x\in \Omega} \calL(x,z_{k+1})\geq 
    \!-\!\min_{x\in \Omega_{k}} \calL(x,z_{k+1}).
\end{equation} 
Then, if \cref{eq:test} is satisfied, the true cost drop $g(y_k) - g(z_{k+1})$ would be at least $\beta$ close to the approximated dual cost drop. In this case, we have found a good candidate $(w_{k+1},z_{k+1})$, and we update the iterates $x_{k+1} = w_{k+1}$ and $y_{k+1} = z_{k+1}$ (this is called as a \emph{descent step}). Otherwise, we set $x_{k+1} = w_{k}$ and $y_{k+1} = z_{k}$ (this is called as a \emph{null step}).

\vspace{3pt}

\noindent
\textbf{Step c) Updating the inner approximation.} Regardless of a descent or null step, we update $\Omega_{k+1}$ as follows.  

\vspace{-1mm}

\begin{assumption}
    \label{assump:AALM}  
    The approximation set $\Omega_{k+1}$ satisfies \vspace{-3mm}
    
    \begin{enumerate}[leftmargin=*]
    \setlength{\itemsep}{0pt}
    \item \textbf{Proper set:} $\Omega_{k+1}$ is convex and closed;
        \item \textbf{Inner approximation:} we have $\Omega_{k+1} \subseteq \Omega $; \vspace{-0.5mm}
        \item \textbf{Dual information:} we require $v_{k+1} \in \Omega_{k+1}$, where $v_{k+1} \in \Omega$ satisfies $g(z_{k+1}) = - \calL(v_{k+1},z_{k+1})$; 
        \item \textbf{Primal information:} if the step $k$ is a null step, then we require $ w_{k+1} \in \Omega_{k+1}$.
    \end{enumerate}
\end{assumption}

\vspace{-2mm}

Essentially, we use a \textit{bundle} of past points, including $v_{k+1}$ (from its dual) and $w_{k+1}$ (from~its~primal), to construct the inner approximation. These rules in updating $\Omega_{k+1}$ are crucial for the convergence guarantees, and they share a similar structure with \Cref{assump:bm}. As alluded to at the beginning of this section, the set $\Omega_{k+1}$ can simply contain only a line segment connecting $v_{k+1}$ and $w_{k+1}$, i.e.,
\begin{align*}
    \Omega_{k+1} =  \{\alpha v_{k+1} +  (1-\alpha) w_{k+1} : 0  \leq \alpha \leq 1 \}.
\end{align*}
In this case, the subproblem \cref{eq:ALM-exact-set-min-main-text} can be extremely simple to solve, as it becomes a convex quadratic program (QP) with a scalar variable ({see \Cref{appendix:computations-of-BALA} for computational details}).

\subsection{Our algorithm: \alg}

Following the bundle-based approximation~above,~we~list our \underline{B}undle-based \underline{A}ugumented \underline{L}agrangian \underline{A}lgorithm (\alg) in \Cref{alg:bundle-dual}. 
Our algorithm \alg{} is motivated by and closely related~to~the classical iALM (\Cref{alg:alm}) and PBM (\Cref{alg:bundle}). 
Meanwhile, \alg{} has some unique features. Compared with iALM, \alg{} can be viewed as a single-loop algorithm (nested two loops are unnecessary), and {all the iterations in \alg{} can be executed efficiently}; we provide further discussions on the computations of each step in \alg{} in \Cref{appendix:computations-of-BALA}. 

\begin{algorithm}[t]
\caption{\underline{B}undle-based \underline{A}ug. \underline{L}agrangian \underline{A}lg. (\alg)
}
\label{alg:bundle-dual}
\begin{algorithmic}[1]
\Require $\Omega_1 \subseteq \Omega, x_1 \in \RR^n ,y_1 \in \RR^m, \rho  > 0, k_{\max} \in \mathbb{N}.$
\For{$k=1,2, \ldots, k_{\max}$}
    \State Acquire a primal candidate via solving \cref{eq:ALM-exact-set-min-main-text} exactly. 
    \State Compute a dual candidate via \cref{eq:dual-update-ALM}. 
    \If{\eqref{eq:test} is satisfied}
     \State Set $x_{k+1} = w_{k+1}, y_{k+1} = z_{k+1}$.  \hfill \Comment{{descent step}}
    \Else
       \State Set $x_{k+1} = w_{k}, y_{k+1} = y_k$. \hfill  \Comment{\textit{null step}}
    \EndIf
    \State Update the set $\Omega_{k+1}$ satisfying \cref{assump:AALM}.
\EndFor
\end{algorithmic}
\end{algorithm}
Compared with PBM, our main focus is on the inner approximation of the convex set $\Omega$ in the ALM framework, while PBM focuses on the lower approximation of the convex function $f$ in the PPM framework. 
We present sublinear and linear convergence guarantees in \cref{section:convergence}. We give a more in-depth discussion of the primal and dual principles in \alg{} with respect to iALM and PBM in \Cref{section:AALA-PBM}. Our algorithm \alg{} is also closely related to the standard subgradient method and its variants. These discussions are provided in \Cref{appendix:subgradient-methods}. 

\section{Convergence Guarantees of \alg} 
\label{section:convergence}
In this section, we establish the sublinear and linear convergence guarantees of \alg. Throughout, we assume Slater's condition for both \eqref{eq:primal} and \eqref{eq:dual}. Thus, strong duality holds, and \eqref{eq:primal} and \eqref{eq:dual} are solvable. We also assume $\Omega$ is compact so that its dual function is Lipschitz continuous. 

\subsection{Progress of descent steps}

\alg{} is a primal and dual algorithm and our analysis also heavily depends on primal and dual connections.  
Here, we present a few lemmas connecting primal and dual progress for descent steps. Our first result is an implication of \cref{eq:test}.

\begin{lemma}[Implication of \cref{eq:test}]
    \label{prop:connection}
    Fix any $\rho > 0$. If step $k$ of \alg{} satisfies the criterion \cref{eq:test}, then we have
    \begin{equation*}
        \calL_{\rho}(x_{k+1},y_k) - \min_{x \in \Omega} \calL_{\rho}(x,y_k)\leq (g(y_k) - g(y_{k+1}))/\beta.
    \end{equation*}
\end{lemma}

The proof is a direct application of the inequality in \cref{eq:g-gk}; see details in \cref{apx:prop:connection}. \Cref{prop:connection} provides an interesting connection between \alg{} and iALM: for a descent step, $x_{k+1}$ is an $\epsilon_k$-solution to the subproblem $\min_{x \in \Omega} \calL_{\rho}(x,y_k)$ with $\epsilon_k = (g(y_k) - g(y_{k+1}))/\beta$. Further discussions between \alg{} and iALM are provided later in \cref{subsection:AALM-iALM}.

We can also quantify the step length, primal residual, and cost value gap at any descent step. 
\begin{lemma}[Step length and primal residual]
    \label{lemma:bundle-dual-consecutive}
       Fix any $\rho > 0$. If step $k$ of \alg~satisfies the criterion \cref{eq:test}, then we have 
     \begin{subequations}
    \begin{align}
          \|y_{k+1} - y_k\|^2/(2\rho)    &\leq (g(y_{k})-\gstar)/\beta, \label{eq:bound-dual} \\
         \rho \| \Amap x_{k+1} - b\|^2/2 & \leq  (g(y_{k}) - \gstar)/\beta.
        \label{eq:bound-affine}
    \end{align}
     \end{subequations}
\end{lemma}

\begin{lemma}[Primal cost value]
    \label{lemma:cost-value}
     Fix any $\rho > 0$. If step $k$ of \alg~satisfies the criterion \cref{eq:test}, then we have 
    \begin{align*}
        -\|\ystar\|\|\Amap x_{k+1} -  b\| \leq \innerproduct{c}{x_{k+1}} - p^\star  
          \leq   2   (g(y_k) - \gstar)/\beta   + \rho^2 \|y_k\|  \|\Amap x_{k+1} - \!b\|,
    \end{align*}
        where $\ystar$ is any dual optimal solution in \eqref{eq:dual}.
\end{lemma}

It is clear that \cref{eq:bound-dual} bounds the step length of dual iterates, \cref{eq:bound-affine} bounds the primal residual, and \Cref{lemma:cost-value} bounds the primal cost gap. The proof of \Cref{lemma:cost-value,lemma:bundle-dual-consecutive} is not difficult; we provide the details in \cref{apx:lemma:bundle-dual-consecutive,apx:lemma:cost-value}. 

\subsection{Sublinear convergences for convex problems}
\label{subsection:sublinear}

\Cref{prop:connection,lemma:bundle-dual-consecutive,lemma:cost-value} quantifies the progress of descent steps. One remaining point is to establish that~the number of null steps between two descent steps is upper bounded. This is indeed the case~by~realizing that \alg{} has a proximal bundle interpretation in the dual (see \cref{subsection:AALA-PBM}, where we show that \alg{} and PBM in \Cref{alg:bundle} share a primal and dual relationship, analogous to the connection between the exact ALM and PPM \cite{rockafellar1976augmented}). Thus, the number of null steps in our algorithm \alg{} is guaranteed to be bounded, following the analysis in \cref{thm:convergence-bundle}.

\begin{lemma} \label{lemma:null-steps}
Given any sub-optimality measure $\epsilon > 0$. The number of null steps between two~consecutive descent steps in \alg{} is bounded by $\bigO(\epsilon^{-2})$.   
\end{lemma}

We now have our first asymptotic convergence guarantee. 

\begin{theorem}[Asymptotic convergence]
    \label{theorem:asymptotic}
    Fix any $\rho > 0$. Let $\{x_k,y_k\}$ be a sequence from \alg{}. We~have 
    \begin{enumerate}[leftmargin=*]
        \item the sequence $\{y_k\}$ converges to an optimal dual solution,
        \item  $\lim_{k\to \infty} \|\Amap x_{k} - b\| = 0,\lim_{k\to \infty} \innerproduct{c}{x_{k}} - \pstar = 0 $. 
    \end{enumerate}
\end{theorem}

The proof of \cref{lemma:null-steps} simply requires arguing that the accumulative inexactness provided in \cref{prop:connection} is summable, and then the result follows from \cref{thm:ALM-convergence}. We provide the details in \cref{apx:theorem:asymptotic}. Aside from the asymptotic result, \alg{} also has a sublinear convergence rate.
\begin{theorem}[Sublinear convergences]
    \label{theorem:convergence-bundle-dual}
    For any $\epsilon > 0$, \alg~with parameters $\beta \in (0,1)$ and $\rho > 0$ finds a dual iterate $y_k$ satisfying $g(y_k) - g^\star \leq \epsilon$ in at most $ \bigO \left (\epsilon^{-3}\right )$ number of iterations, and a primal iterate $x_k$ satisfying
    $
            |\innerproduct{c}{x_k} - \pstar| \leq \epsilon\; \text{and}\; \|\Amap x_k - b\| \leq \epsilon
    $
    in at most $ \bigO \left ( \epsilon^{-6} \right )$ number of iterations. Moreover, if we choose $\rho = \epsilon$, then the iteration complexities are improved to $\bigO(\epsilon^{-2})$ and $\bigO(\epsilon^{-4})$ for the dual and primal iterates, respectively. 
\end{theorem}
The proof is provided in \cref{apx:theorem:convergence-bundle-dual}. Our result in \cref{theorem:convergence-bundle-dual} resembles the findings of bundle methods in \cite[Lemma 3.4-3.6]{ding2023revisiting} in the sense that the dual cost gap controls the affine feasibility and the primal cost gap. Despite the similarity, our results in \cref{theorem:asymptotic} further confirm the convergence of the sequence $\{y_k\}$. We note that the convergence of $\{y_k\}$ can also be established from bundle method analysis, e.g. \cite[Theorem 7.16]{ruszczynski2011nonlinear}. However, our analysis is more straightforward by arguing the accumulative inexactness is bounded and, thus, the convergence of dual iterates follows from classical results in \cite{rockafellar1976augmented}.

In \Cref{theorem:convergence-bundle-dual}, the complexity of the primal residuals $\| \Amap x_k - b\|$ and $|\innerproduct{c}{x_k} - \pstar|$ is a square-order slower than the dual cost gap. This is because the analysis does not fully utilize the connection between \alg{} and iALM. Let $\calS_k = \{ j \in \mathbb{N} :\text{step } j \text{ is a descent step}, j \leq k \} \cup \{0\}$ be the set of indices for descent steps before iteration $k$. We establish an improved rate for average iterates here.
\begin{theorem}[Average iterates] 
    \label{thm:average-iterate}
     For any $\epsilon > 0$, \alg{} with an initialization $y_1 = 0$ and parameters $\beta \in  (0,1)$ and $\rho  >  0$ finds an average iterate $\Bar{x}_k  =  \sum_{j \in \calS_k} \bar{x}_{j+1} /|\calS_k| $ satisfying $ |\Amap  \Bar{x}_k - b\|\leq \epsilon$ and $|\innerproduct{c}{\Bar{x}_k} - \pstar |  \leq \epsilon$ in at most $\bigO(\epsilon^{-3})$ number of iterations. 
\end{theorem}
The proof leverages the relation between \alg{} and iALM, as well as a recent result in \cite[Theorem 4]{xu2021iteration}; we provide the proof details in \cref{apx:thm:average-iterate}. {We remark that since \alg{} has a dual interpretation of PBM (see \cref{subsection:AALA-PBM}), \cref{thm:average-iterate} also holds for the average iterates in PBM. To our best knowledge, \cref{thm:average-iterate} is the first convergence result on average iterates of PBM, which gives a big improvement in iteration complexity. Precisely, the primal average iterate $\Bar{x}_k$ improves the iteration complexity from $\bigO(\epsilon^{-6})$ to $\bigO(\epsilon^{-3})$, compared with  the last iterate guarantee in \cref{theorem:convergence-bundle-dual}.} 

    \begin{remark}[Comparison with existing works]
    We compare \alg{} with other recent works in the ALM framework \cite{yurtsever2019conditional,yurtsever2021scalable,garber2023faster}. 
   Our algorithm \alg{} has sublinear convergences for any fixed parameter $\rho>0$. However, the convergence of \CGAL{} in \cite{yurtsever2019conditional,yurtsever2021scalable} requires an increasing penalty term $\rho_k$ that approaches infinity (in this sense, \CGAL{} is more like a penalty method, instead of ALM). As a result, \CGAL{} has no guarantee of the convergence of dual sequence $\{y_k\}$. In contrast, our \alg{} guarantees the convergence of the dual sequence as  $ \bigO \left (\epsilon^{-3}\right )$. More recently, \cite{garber2023faster} refined \CGAL{}  and improved the iteration complexity for the primal residuals to $\bigO(\epsilon^{-1})$ under stronger assumptions. Still, the convergence of the dual sequence $\{y_k\}$ remains unknown in \cite{garber2023faster}. We will next establish that \alg{} can achieve linear convergence under certain regularity conditions. In those cases, \alg{} outperforms \CGAL{} and its variants significantly; we also observe these improvements in our numerical experiments.   \hfill $\square$
    \end{remark}

\subsection{Linear convergences under regularity conditions}
\label{subsection:linear}
We have established the sublinear convergences of \alg{}  when the construction of $\Omega_k$ satisfies \cref{assump:AALM}. We here show linear convergence for \alg{} under a stronger assumption on $\Omega_k$ and the dual function $g$ defined in \eqref{eq:dual-minimizatio}. 

Our intuition is inspired by the recent advances in linear convergences in inexact ALM and PPM \cite{cui2019r,liao2024inexact,liao2024error}. All these results require a \textit{quadratic growth} condition, i.e., there is a constant $\alpha > 0$ such that 
\begin{equation}
\label{eq:QG}
    g(y) - g^\star \geq \frac{\alpha}{2} \cdot  \Dist^2(y,\dsolution), \quad \forall y \in \mathbb{R}^m, 
\end{equation}
where $\dsolution$ is the optimal solution set in \cref{eq:dual-minimizatio}. As our \alg~is closely related to iALM and PBM (see \cref{section:AALA-PBM}), we may thus expect linear convergence when $g$ satisfies quadratic growth \cref{eq:QG}. However, the iALM assumes an oracle to obtain an inexact point $x_{k+1}$ satisfying \eqref{eq:ALM-inexact-min}. In contrast, \alg{} is single-loop and incorporates all the intermediate steps to update the next point $x_{k+1}$ (i.e., null steps). We need another key assumption: there exists a $\gamma > 0$, such that the approximated dual function $g_k$ in \eqref{eq:approximated-dual} satisfies 
\begin{equation}
    \label{eq:quadratic-close}
    g_k(y) \leq  g(y)  \leq g_k(y) + \frac{\gamma}{2}\|y-y_k\|^2,\; \forall y \in \RR^m.
\end{equation}
This condition means that the lower approximation function $g_k$ captures the true function $g$ well up to a quadratic error. 

\begin{theorem}[Linear convergences] 
    \label{thm:linear-convergence}
    Suppose that the dual function $g$ satisfies \eqref{eq:QG} and the approximation function $g_k$ in \eqref{eq:approximated-dual} satisfies \eqref{eq:quadratic-close} for all $k\geq T$ with a constant $\gamma > 0$. Choose $\beta \in (0,1/2]$ and an augmented Lagrangian constant $\rho\geq \frac{1}{\gamma}$ in \alg{}. For all iterations $k \geq T$, we~have:     
    \begin{enumerate}[leftmargin=*]
    \setlength{\itemsep}{0pt}
        \item Every iteration is a descent step, i.e., \eqref{eq:test} is satisfied;
        \item There exists two constants $\mu_1 \in (0,1), \mu_2 > 0$ such that  $\Dist(y_{k+1}, \dsolution) \leq \mu_1 \cdot \Dist(y_{k}, \dsolution)$ and $\max\{g(y_k) - g^\star, \|\Amap x_{k} - b \|^2, | \innerproduct{c}{x_k} - \pstar |^2\}  \leq \mu_2  \cdot  \Dist(y_{k}, \dsolution)$. 
    \end{enumerate}
\end{theorem}

The proof is given in \cref{apx:thm:linear-convergence}, which relies on an abstraction of the analysis techniques in \cite{ding2023revisiting,liao2023overview} that focus on semidefinite programs (SDPs). Indeed, our algorithm \alg{} can be directly applied to SDPs, where all steps in \Cref{alg:bundle-dual} have explicit constructions. We outline the details of \alg{} when applied to solving SDPs in \Cref{section-application-SDP}. In this case, \alg{} closely resembles the classical spectral bundle method in \cite{helmberg2000spectral}; see \cite{ding2023revisiting,liao2023overview} for recent discussions. 
 
We further note that the quadratic growth \cref{eq:QG} holds for a class of SDPs \cite{ding2023revisiting,liao2024inexact} that arise in a range of applications, e.g., Max-Cut \cite{goemans1995improved} and matrix optimization \cite{cui2017quadratic}. The quadratic closeness condition \eqref{eq:quadratic-close} is also true when  $\Omega_k$ captures the rank of an optimal solution in SDPs \cite{ding2023revisiting}. In~particular, if an SDP admits a rank-one optimal solution, the set $\Omega_k$ can be constructed in a simple set while ensuring \eqref{eq:quadratic-close}. Thus, for SDPs with rank-one solutions, our algorithm \alg{} enjoys fast linear convergences (guaranteed by \cref{thm:linear-convergence}) while admitting extremely efficient updates (guaranteed by the simple subproblem \cref{eq:ALM-exact-set-min-main-text} with an analytical solution). Our algorithm significantly outperforms the conditional-gradient Augmented Lagrangian (\CGAL) method in \cite{yurtsever2021scalable,yurtsever2019conditional}.

\section{Numerical experiments}
\label{section:experiment}
We here demonstrate the numerical performance of \alg. We compare \alg~against \CGAL~\cite{yurtsever2019conditional},~since both \alg~and \CGAL~are designed within the ALM framework. Moreover, \CGAL~has~been~used to design a scalable solver for SDPs \cite{yurtsever2021scalable}. The details of applying \alg{} to SDPs are given in \Cref{section-application-SDP}. More extensive numerical results are provided in \Cref{apx:section:detail-numerical}.

\begin{figure}[t]
\setlength\textfloatsep{0pt}
\captionsetup{skip=0pt}
     \centering \vspace{-3mm}
{\includegraphics[width=1\textwidth]
{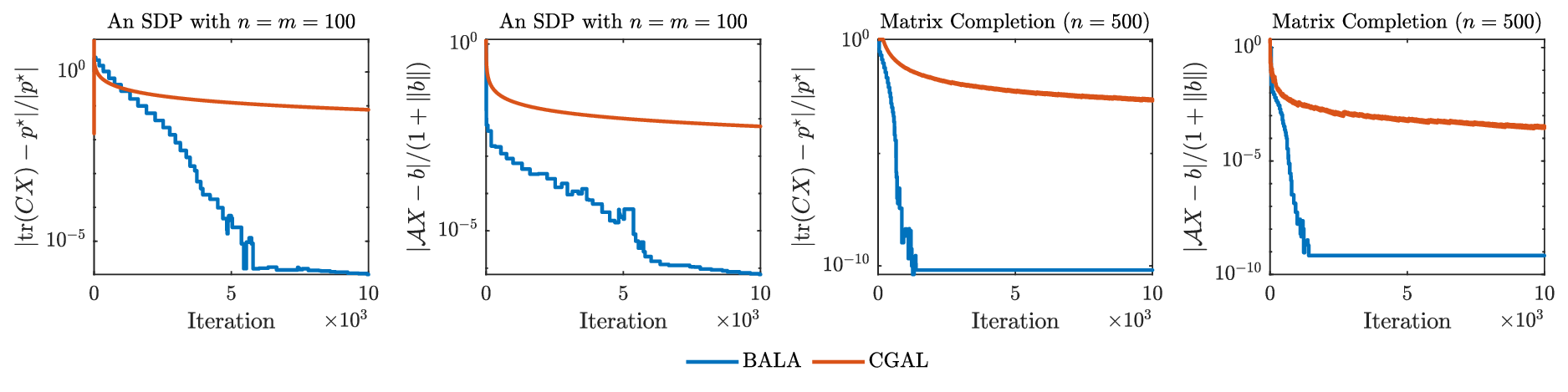}}
\caption{Comparison between our proposed \alg{} (\Cref{alg:bundle-dual}) and \CGAL \cite{yurtsever2019conditional,yurtsever2021scalable}.} \vspace{-3mm}
\label{fig:numerical-experiment}
\end{figure}

We here consider two SDPs of the form
\begin{align*}
    \min_{X \in \sn} \{ \Trace(CX) : \Amap X = b, \Trace(X) \leq a ,X \in \snp\},
\end{align*}
where $C \in \sn, \Amap : \sn \to \RR^m$ and $ a \geq 0$ are the problem data.~This type of SDP is in the form of \eqref{eq:primal} with the compact set $\Omega = \{ X \in \snp : \Trace(X) \leq a \}$. Our first instance is a SDP  with randomly generated data, where the dimension is $n=m=100$. The second SDP is a convex relaxation for matrix completion \cite{candes2012exact} with dimensions being $n = 500$ and $m = 12414$. More experimental details are provided in \cref{apx:section:detail-numerical}.

For a fair comparison, we construct the set $\Omega_{k}$ as $ \{ \alpha v_{k} + \beta w_{k} : \alpha \geq 0, \beta \geq 0, \alpha + \beta \leq 1 \}$ for all $k$ and generate the SDPs that admit at least a rank-one solution. 
In this case, the approximated dual function captures the true dual function up to a quadratic term \cite{ding2023revisiting,liao2023overview}, and the subproblem can also be solved analytically. Therefore, as guaranteed by \Cref{thm:linear-convergence}, we expect fast convergences for \alg{}. 
As shown in \cref{fig:numerical-experiment}, \alg{} indeed significantly outperforms \CGAL{} in both randomly generated SDP and the matrix competition problem. In the former case, \alg{} quickly reaches the accuracy of $10^{-5}$, while \CGAL{} suffers from slow convergence and can only obtain the accuracy of $10^{-2}$ after 10000 iterations. 
In the case of matrix competition, \alg{} even achieves a $10^{-9}$ accuracy, while \CGAL{} only achieves $10^{-3}$ accuracy after 10000 iterations. 

In \cref{apx:section:detail-numerical}, we show that \alg{} achieves superior performance in many other SDPs, including those from binary quadratic and polynomial optimization problems. 

\section{Conclusion}
\label{section:conclusion} \vspace{-1mm}
In this paper, we have introduced a novel bundled-based Augmented Lagrangian framework and~a new algorithm  \alg{}. Our algorithm enjoys a sublinear convergence under a mild assumption and achieves fast linear convergence when certain regularity conditions are satisfied. The superior~performance of \alg{} is demonstrated in numerical experiments. We are interested in applying \alg{} to solve large-scale conic programs. The connections between \alg{} with other first-order algorithms, such as momentum-based and Frank-Wolfe-type algorithms, are also interesting.

{\small 
\bibliographystyle{unsrt}
\bibliography{reference.bib}
}

\small 
\newpage
\appendix
\numberwithin{equation}{section}
\section*{Appendix}

\textbf{Organization} In this appendix, we provide extra discussions, proof details, and further computational results. In particular, we organize this appendix as follows:

\vspace{-2mm}

\begin{itemize}
\setlength{\itemsep}{0pt}
\item \cref{section:computation} discusses the computation for solving the subproblem in \alg{};
\item \cref{section:AALA-PBM} establishes connections of \alg{} to inexact ALM and proximal bundle method;
\item \cref{appendix:subgradient-methods} discusses the connection of \alg{} to the dual (sub)gradient method;
\item \cref{apx:section:missing-proof} completes the missing technical proofs in \cref{section:convergence}; 
\item \cref{section-application-SDP} shows that application of \alg{} in semidefinite programs;
\item \cref{apx:section:illustration} illustrates the procedure of \alg{} using a simple example;
\item \cref{apx:section:detail-numerical} details the numerical experiment in \cref{section:experiment} and presents additional experiments. 
\end{itemize}

\section{Computations in \alg} \label{appendix:computations-of-BALA}
\label{section:computation}
We here discuss computation details in our proposed \alg{} (see \cref{alg:bundle-dual}). At every iteration $k$, \alg{} performs three updates: a) computing a pair of primal and dual candidates from \cref{eq:ALM-exact-set-min-main-text}; b) testing the descent progress; c) updating the approximation $\Omega_{k+1}$. Typically, computing the primal candidate $w_{k+1}$ is the most computationally intensive step, which needs to solve the subproblem \cref{eq:ALM-exact-set-min-main-text}. 
The update of a dual candidate from \cref{eq:dual-update-ALM} is computationally trivial. 

\subsection{Analytical solution to the subproblem \cref{eq:ALM-exact-set-min-main-text}}

For convenience, we reproduce the subproblem \cref{eq:ALM-exact-set-min-main-text} below 
\begin{equation}
 \label{eq:ALM-exact-set-min-app}
    w_{k+1} \in \argmin_{x \in \Omega_{k}}\; \calL_{\rho}(x,y_k) = \innerproduct{c}{x} + \innerproduct{y_k}{b-\Amap x} + \frac{\rho}{2}\|b - \Amap x \|^2. 
\end{equation}
This is a convex problem with a quadratic cost function. Since the cost function is quadratic, the complexity of solving \cref{eq:ALM-exact-set-min-app} largely depends on the constraint set $\Omega_{k}$. Our idea of \alg{} is to design a simple constraint set $\Omega_k$ so that \cref{eq:ALM-exact-set-min-app} can be solved efficiently. Our requirement regarding the update of $\Omega_{k}$ is listed in \cref{assump:AALM}.

In particular, the inner approximations $\Omega_{k}$ should contain a bundle of past iterates to approximate $\Omega$. As listed in \cref{assump:AALM}, it can be as simple as a line segment, including $v_{k}$ (from its dual) and $w_{k}$ (from its primal), i.e., 
$$
\Omega_k = \mathrm{conv}(v_k,w_{k}) := \{\alpha  v_k + (1-\alpha)  w_k :  0 \leq \alpha \leq 1 \}.
$$ In this case, the subproblem \cref{eq:ALM-exact-set-min-app} becomes an equality-constrained quadratic program with a scalar variable $0 \leq \alpha \leq 1$. Its analytical solution is given by a saturation operator, $\alpha^\star = \mathrm{sat}(\phi)$, given below   
\begin{align*}
   \mathrm{sat}(\phi) = \begin{cases}
        0, &\text{if } \phi < 0, \\
        \phi, &\text{if }  0 \leq \phi \leq 1,  \\
        1, &\text{if } \phi > 1,
    \end{cases}
\end{align*}
where we have defined 
$$
\phi = \frac{\rho \innerproduct{\Ajmap (b-\Amap w_{k})}{v_k - w_k} + \innerproduct{\Ajmap y_k - c }{ v_k - w_k} }{\rho \|\Amap (v_k - w_k)\|^2}.
$$

As shown in \Cref{theorem:asymptotic,theorem:convergence-bundle-dual,thm:average-iterate}, this simple update already guarantees the asymptotic and sublinear convergences of our algorithm \alg{}. Depending on different applications, one may use a few more past iterates to construct the inner approximations $\Omega_k$, which may give a better convergence behavior. Indeed, as we prove in \cref{thm:linear-convergence},  \alg{} enjoy linear convergences if 1) the dual optimization satisfies a \textit{quadratic growth} condition \cref{eq:QG}, and 2) the inner approximation $\Omega_k$ is chosen such that the induced dual approximation is \textit{quadratically close} to the original dual function \cref{eq:quadratic-close}. In this case, one may want to choose a few more points to construct $\Omega_k$ that requires slightly more computation to solve \cref{eq:ALM-exact-set-min-app}. 

For example, in SDP applications, if an optimal solution is of low rank, a good strategy is to design an inner approximation $\Omega_k$ that captures the rank of the optimal solutions \cite{ding2023revisiting,liao2023overview}. In this situation, the subproblem can also be solved efficiently as it only involves a small semidefinite constraint. Moreover, linear convergence is also guaranteed under mild assumptions (e.g., strict complementarity). 
We present further details of applying \alg{} to solve SDPs in \Cref{section-application-SDP}.

\subsection{Computations in the descent progress test and inner approximation update}
\label{apx:subsection:compuation}
Aside from the main computation of solving the subproblem \cref{eq:ALM-exact-set-min-main-text}, we discuss the computation for the test \eqref{eq:test} and update the set $\Omega_{k+1}$. We first discuss the computation for the test \eqref{eq:test}. At iteration $k$, we recall that the test requires the knowledge of the three values $g(y_k)$, $g(z_{k+1}),$ and $g_k(z_{k+1})$. In general, obtaining $g(y_k)$ and $g(z_{k+1})$ requires solving an optimization problem and the complexity depends on the difficulty of $\Omega$. However, when dual function $g$ admits an analytical form, obtaining $g(y_k)$ and $g(z_{k+1})$ simply requires two evaluations of the function. This is indeed the case when the compact set $\Omega$ is well-structured. For example, three common symmetric cones (nonnegative orthant, second-order cone, and positive semidefinite cone) with the natural bound admit this property.
\begin{itemize}[leftmargin=*]
    \item (Nonnegative orthant) Let $\Omega =\{x \in \RR^n_+ : \|x\|_1 \leq a\}$ with $a \geq 0$. The dual function $g$ can be shown as \begin{equation*}
    g(y) = -\innerproduct{b}{y} + a \max\{(\Ajmap y -c)_1,\ldots,(\Ajmap y -c)_n ,0 \},
\end{equation*}
where $(\Ajmap y -c)_i$ denotes the $i$-th element of the vector $\Ajmap y -c$.
    \item (Second-order cone) Let  $\Omega = \{x \in \RR^{n+1} : \|x_{1:n}\|\leq x_{n+1}, x_{n+1} \leq a \}$ with $a \geq 0$. The dual function $g$ takes the form 
    \begin{equation*}
    g(y) = -\innerproduct{b}{y} + a \max\{ \| (\Ajmap y -c)_{1:n} \| - (\Ajmap y -c)_{n+1},0\},
\end{equation*}
where $(\Ajmap y -c)_{1:n}$ denotes the first $n$ elements of $\Ajmap y -c$. 
    \item (Positive semidefinite cone) Let $ \Omega=  \{x\in \snp : \Trace(x) \leq a \}$ with $a \geq 0$. The dual function $g$ becomes 
    \begin{equation*}
        g(y) = -\innerproduct{b}{y} + a \max\{\lambda_{\max}(\Ajmap y - c),0\}.
    \end{equation*}
\end{itemize}
In addition, if the step $k-1$ is a null step, then we have $g(y_{k}) = g(y_{k-1})$. If the step $k-1$ is a descent step, then we have $g(y_{k}) = g(z_{k})$. Thus, one can keep track of progress to save computations. 

On the other hand, the function value $g_k(z_{k+1})$ can be retrieved readily once the subproblem \eqref{eq:ALM-exact-set-min-main-text} is solved. In particular, as we will see in \eqref{eq:prox-gkj}, it holds that
\begin{align*}
     -\min_{x \in \Omega_{k}}  \calL_\rho(x,y_k) = \min_{y \in \RR^m} g_{k}(y) + \frac{1}{2\rho}\|y- y_k\|^2.
\end{align*}
Thus, once $w_{k+1} \in \min_{x \in \Omega_{k}}\calL_\rho(x,y_k)$ and $z_{k+1} = y_k + \rho (b- \Amap w_{k+1})$ are constructed, $g_k(z_{k+1})$ can be recovered as 
\begin{align*}
    g_k(z_{k+1}) = -\calL_\rho(w_{k+1},y_k) - \frac{1}{2\rho}\|z_{k+1} - y_k\|^2,
\end{align*}
which is computationally trivial.

We next discuss the complexity of the update of $\Omega_{k}$. Viewing \cref{assump:AALM}, we see that the only part that requires computation is to find a point $v_{k+1} \in \Omega_{k+1}$ such that $g(v_{k+1}) = - \calL(v_{k+1},z_{k+1})$. In other words, it requires finding a point $v_{k+1}$ that realizes the dual function value in the Lagrangian function. In general, this requires solving an optimization problem. However, in some cases, it can be done by computing the eigenvector of $\Ajmap z_{k+1} -c$. Specifically, consider the above three concrete examples of nonnegative orthant, second-order cone, and positive semidefinite cone. It can be shown that $v_{k+1}$ can be computed in the following ways: 
\begin{itemize}[leftmargin=*]
    \item (Nonnegative orthant) Let $p(y) = \max\{(\Ajmap y -c)_1,\ldots,(\Ajmap y -c)_n ,0 \}$. We have
    \begin{align*}
        v_{k+1} = \begin{cases}
            \mathbf{0},  & \text{if}~  p(z_{k+1}) = 0,\\
            a e_i,       & \text{if}~p(z_{k+1}) > 0, 
        \end{cases}
    \end{align*}
    where $\mathbf{0} \in \RR^n$ is a zero vector, and $e_i$ is the standard basis in $\RR^n$ such that $(\Ajmap z_{k+1} -c)_i = p(z_{k+1})$.
    \item (Second-order cone) Let $p(y) = \max\{ \| (\Ajmap y -c)_{1:n} \| - (\Ajmap y -c)_{n+1},0\}$. We have 
    \begin{align*}
        v_{k+1} = \begin{cases}
            \mathbf{0},  & \text{if}~  p(z_{k+1}) = 0,\\
            a v_+,       & \text{if}~p(z_{k+1}) > 0, 
        \end{cases}
    \end{align*}
    where $\mathbf{0} \in \RR^{n+1}$ is a zero vector, and $v_+$ is the eigenvector of $\Ajmap z_{k+1} - c$ corresponding to the eigenvalue $ \| (\Ajmap y -c)_{1:n} \| - (\Ajmap y -c)_{n+1}$ in the sense of the second-order cone. We refer readers to \cite{alizadeh2003second} for more details regarding the eigenvalue/eigenvector in the second-order cone.
    \item (Positive semidefinite cone) Let $p(x) =\max\{\lambda_{\max}(\Ajmap y - C),0\}$. We have 
    \begin{align*}
        v_{k+1} = \begin{cases}
            \mathbf{0},  & \text{if}~  p(z_{k+1}) = 0,\\
            a v_+ v_+^\tr,       & \text{if}~p(z_{k+1}) > 0, 
        \end{cases}
    \end{align*}
    where $\mathbf{0} \in \sn$ is a zero matrix, and $v_+ \in \RR^n$ is the eigenvector of $\Ajmap z_{k+1} - c$ such that $v_+^\tr (\Ajmap z_{k+1} - c) v_+ = \lambda_{\max}(\Ajmap z_{k+1} - c)$.
\end{itemize}

\section{Primal-dual Principles in \alg: Interplay with PBM and iALM}
\label{section:AALA-PBM}
In this section, we discuss the relationship among \alg{} (\Cref{alg:bundle-dual}), PBM (\Cref{alg:bundle}), and iALM (\Cref{alg:alm}). In \Cref{subsection:AALA-PBM}, we reveal 
 a deep primal and dual relationship between \alg{} in \Cref{alg:bundle-dual} and the PBM in \Cref{alg:bundle}, analogous to the connection between the exact ALM and PPM revealed in \cite{rockafellar1976augmented}. In \Cref{subsection:AALM-iALM}, we demonstrate that \alg~is also closely related to the iALM in \Cref{alg:alm}. 

\subsection{Primal-Dual interplay between \alg{} and PBM}
\label{subsection:AALA-PBM}
In this subsection, we show that running \alg{} is equivalent to applying PBM for the dual problem \eqref{eq:dual}. This is analogous to the connection in \cite{rockafellar1976augmented}: running an exact ALM on the primal is equivalent to executing an exact PPM on its dual. With this connection, many proofs in the main text (including \Cref{theorem:asymptotic,theorem:convergence-bundle-dual}) are simplified and can be adapted from the PBM counterparts (note that the proof of \cref{thm:average-iterate} relies on a recent result in the ALM framework \cite{xu2021iteration} and has not been established in the literature on PBM). This is also the same as a classical strategy for analyzing (exact or inexact) ALM since the seminal work \cite{rockafellar1976augmented}.

We first present an important lemma that connects the dual update in \alg{} with a proximal step. 

\begin{lemma} \label{eq:dual-bala-proximal}
    The dual update $z_{k+1}$ from \cref{eq:dual-update-ALM} in \alg{} is the same as the proximal step on the induced approximated dual function, i.e.,
    $$
    z_{k+1} = \min_{y \in \RR^m} g_{k}(y) + \frac{1}{2\rho}\|y- y_k\|^2,
    $$
    where $g_{k}$ is defined in \cref{eq:approximated-dual}.
\end{lemma}
\begin{proof}
Our proof idea is based on primal and dual analysis and the KKT optimality condition. Consider the bundled-based augmented Lagrangian subproblem \eqref{eq:ALM-exact-set-min-main-text} (or equivalently \cref{eq:ALM-exact-set-min-app}) in \alg{} at the iteration $k$. We can reformulate this problem in the following form
\begin{subequations} \label{eq:prox-gkj}
\begin{align} -\min_{x \in \Omega_{k}}  \calL_\rho(x,y_k)   
=&   \max_{x \in \Omega_{k}}  -\calL_\rho(x,y_k)  \label{eq:prox-gkj-a}
     \\
     =& \max_{x \in \Omega_{k}} \min_{y \in \RR^m} -\calL(x,y)+\frac{1}{2\rho}\|y- y_k\|^2 \label{eq:prox-gkj-b} \\
     = & \min_{y \in \RR^m} \max_{x \in \Omega_{k}}   -\calL(x,y)+\frac{1}{2\rho}\|y- y_k\|^2 \label{eq:prox-gkj-c} \\
     = & \min_{y \in \RR^m} g_{k}(y) + \frac{1}{2\rho}\|y- y_k\|^2,  \label{eq:prox-gkj-d}
\end{align}
\end{subequations}
where \cref{eq:prox-gkj-a} is a simple sign switch from minimization to maximization, \cref{eq:prox-gkj-b} is from the definition of the augmented Lagrangian function $\calL_\rho(x,y_k)$ in  \cref{eq:AL}, 
\cref{eq:prox-gkj-c} swaps the minimization and maximization operation since the set $\Omega_{k}$ is compact and convex \cite[Corollary 37.3.2]{rockafellar1997convex}, and finally in the last equality \cref{eq:prox-gkj-d}, we have defined an induced approximated dual function (see \cref{eq:approximated-dual})
$$
g_{k}(y) =  \max_{x \in \Omega_{k}} -\calL(x,y) = - \min_{x \in \Omega_{k}} \calL(x,y), \qquad \forall y \in \mathbb{R}^n. 
$$
Once \eqref{eq:ALM-exact-set-min-app} is solved with a primal candidate $w_{k+1}$, we can compute the optimal solution to \eqref{eq:prox-gkj-d} by  
$y_k + \rho ( b- \Amap w_{k+1})$, which is the same as the construction of the dual candidate $z_{k+1}$ in \alg{} (see \cref{eq:dual-update-ALM}). This is the KKT optimality condition for the inner optimization, i.e., 
$
0 =\Amap x^\star - b +(y^\star- y_k)/\rho \, \Rightarrow \, y^\star = y_k + \rho ( b- \Amap x^\star).
$
This completes the proof.    
\end{proof}

In other words, the construction of the dual candidate $z_{k+1}$ from \cref{eq:dual-update-ALM} in \alg{} is equivalent to performing a proximal step on approximated dual function $g_{k}$ centered at $y_k$. We remark that the equivalence in \cref{eq:dual-bala-proximal} is not entirely surprising since when we have $\Omega_k = \Omega$, it recovers the finding in \cite{rockafellar1976augmented} that an exact ALM step on the primal problem is equivalent to an exact PPM step on its dual.

Another interesting connection is that the induced approximated dual function $g_{k}$ satisfies the requirements in \Cref{assump:bm} for PBM introduced in \cref{subsection:bundle}. We have the following result. 
\begin{theorem}
    \label{prop:AALM-PBM}
    Consider the inner approximation $\Omega_{k+1} \subset \Omega$ in \alg. Let $g_{k+1}(\cdot)\!= \!- \min_{x \in \Omega_{k+1}} \calL(x,\cdot)$ be an approximated dual function of $g$. Then, the fact that $\Omega_{k+1}$ satisfies \cref{assump:AALM} implies that the function $g_{k+1}$ satisfies \cref{assump:bm}.
\end{theorem}

\begin{proof}
     Recall that the points $w_{k+1}$ and $z_{k+1}$ correspond to the primal and dual candidates in \alg{}; see the updates in \cref{eq:ALM-exact-set-min-main-text,eq:dual-update-ALM}. 
     We verify that the requirements in \cref{assump:AALM} imply those in \cref{assump:bm} one by one.

    \begin{itemize}
        \item 
    Firstly, as $\Omega_{k+1} \subseteq \Omega$, it is immediate to see that
        \begin{equation} \label{eq:assumption-1-PBM}
            \begin{aligned}
                g(y) = -\min_{x\in \Omega} \calL(x,y)  \geq -\min_{x\in \Omega_{k+1}} \calL(x,y) = g_{k+1}(y), \qquad \forall y \in \RR^m.    
            \end{aligned}
        \end{equation}
    This indicates that the induced approximation $g_{k+1}$ serves a global lower approximation of the original dual $g$. This also highlights the principle that imposing a restriction on the primal problem results in a relaxation on its dual. 
    \item Secondly, we have that 
        \begin{subequations} \label{eq:assumption2-PBM}
            \begin{align}
            g_{k+1}(y)  = -\min_{x \in \Omega_{k+1} } \calL(x,y) 
             & \geq -\calL(v_{k+1},y) \label{eq:assumption2-a} \\
            &  = -\calL(v_{k+1},z_{k+1}) + \innerproduct{z_{k+1} - y}{b-\Amap v_{k+1} } \label{eq:assumption2-b} \\
            & = g(z_{k+1})+\innerproduct{z_{k+1} - y}{b-\Amap v_{k+1}}, \label{eq:assumption2-c} \\
            & = g(z_{k+1})+\innerproduct{\Amap v_{k+1} - b}{y - z_{k+1}}, \forall y \in \RR^m.  \label{eq:assumption2-d}
            \end{align}
        \end{subequations}
        In \cref{eq:assumption2-a}, we use the fact that $v_{k+1} \in \Omega_{k+1}$ is a feasible solution (recall that $\Omega_{k+1}$ contains $v_{k+1}$ by the third requirement in \Cref{assump:AALM}). The equality in  \cref{eq:assumption2-b} applies the fact that the Lagrangian $\mathcal{L}(x,y)$ is linear in $y$, and the equality in \cref{eq:assumption2-c} uses $\calL(v_{k+1},z_{k+1}) = - g(z_{k+1})$ by the third requirement in \Cref{assump:AALM}. 
        
        Recall that we further have $\Amap v_{k+1} - b \in \partial g(z_{k+1})$; see \Cref{lemma:dual-subgradient}. Thus, \cref{eq:assumption2-PBM} indicates that $g_{k+1}$ includes the subgradient lowerbound of $g$ at $z_{k+1}$. 

    \item Lastly, we have the following inequality   
        \begin{subequations} \label{eq:assumptionc-PBM}
            \begin{align}
                  g_{k+1}(y) 
                 = -\min_{x \in \Omega_{k+1}} \calL(x,y) 
                 & \geq -\calL(w_{k+1},y) \label{eq:assumptionc-a} \\
                & = g_{k}(z_{k+1}) - g_{k}(z_{k+1}) -  \innerproduct{c}{w_{k+1}} - \innerproduct{y}{b-\Amap w_{k+1} } \label{eq:assumptionc-b}\\
                & = g_{k}(z_{k+1}) +  \calL(w_{k+1},z_{k+1}) - \innerproduct{c}{w_{k+1}} - \innerproduct{y}{b-\Amap w_{k+1}} \label{eq:assumptionc-c}\\
                & = g_{k}(z_{k+1}) +\innerproduct{z_{k+1}}{b-\Amap w_{k+1}} - \innerproduct{y}{b-\Amap w_{k+1}} \label{eq:assumptionc-d} \\
                & = g_{k}(z_{k+1}) + \innerproduct{z_{k+1}- y}{b-\Amap w_{k+1}} \label{eq:assumptionc-e} \\
                & = g_{k}(z_{k+1}) + \innerproduct{y - z_{k+1}}{(y_k - z_{k+1})/\rho}, \forall y \in \RR^m. \label{eq:assumptionc-f}
            \end{align}
        \end{subequations}
        In \cref{eq:assumptionc-a}, we have used the fact that $w_{k+1}$ is a feasible point in $\Omega_{k+1}$ (recall that $\Omega_{k+1}$ contains $w_{k+1}$ by the fourth requirement in \Cref{assump:AALM} when step $k$ is a null step). The equality in \cref{eq:assumptionc-b} is a simple reformulation, and \cref{eq:assumptionc-c} is by the definition of $g_k$ and the fact $ w_{k+1} \in \argmin_{x\in \Omega_k} \calL(x,y_k)$ since \cref{eq:prox-gkj-c}. The equality in \cref{eq:assumptionc-d} is due to the simple calculation 
        $$
        \calL(w_{k+1},z_{k+1}) - \innerproduct{c}{w_{k+1}} = \innerproduct{z_{k+1}}{b-\Amap w_{k+1}}. 
        $$
        In \cref{eq:assumptionc-f}, we use the update that 
        $
        z_{k+1} - y_k = \rho(b - \Amap w_{k+1}).
        $ 
        Finally, we note that $(y_k - z_{k+1})/\rho \in \partial g_{k}(z_{k+1})$ is the optimality condition for the problem
        $$z_{k+1} = \argmin_{y \in \RR^m} g_{k}(y) + \frac{1}{2\rho} \|y - y_k\|^2.$$
        This establishes that $g_{k+1}$ aggregates the subgradient of $g_{k}$ at $z_{k+1}$. 
    \end{itemize}

To summarize, the three inequalities \cref{eq:assumption-1-PBM,eq:assumptionc-PBM,eq:assumption2-PBM} confirm the three requirements in \Cref{assump:bm}, respectively. 
\end{proof}

\Cref{prop:AALM-PBM} confirms that the function $g_{k}$ is a valid lower approximation for the dual function $g$ in PBM. Thus, with the view of \eqref{eq:prox-gkj}, we see that the dual candidate $z_{k+1}$ in \alg~is the same as the solution of a proximal subproblem with the approximation function $g_{k}$ in PBM for \eqref{eq:dual}. 
Finally, by the definition of $g_k$ \eqref{eq:approximated-dual}, we observe that the testing criterion \eqref{eq:test} in \alg~is the same as that in PMB for \eqref{eq:dual}. Thus, we conclude that \alg~is equivalent to PBM for \eqref{eq:dual} with the sequence of lower approximation functions $\{g_k\}$ induced by the sequence of inner approximation sets $\{\Omega_k\}$ from \alg. 

\vspace{2mm}

\begin{remark}[Insights beyond PBM]
    While \Cref{eq:dual-bala-proximal,prop:AALM-PBM} present an exact primal and dual relation between our algorithm \alg{} and the well-established PBM,  we believe that our bundle-based Augmented Lagrangian framework provides complementary insights and extends beyond the proximal bundle framework. First, as we will demonstrate in \Cref{subsection:AALM-iALM},  our algorithm also has a close relationship with the inexact ALM framework. Then, all classical developments in inexact ALM may be adapted into our algorithmic framework. Second, our algorithm \alg{} essentially uses a line segment between $v_{k+1}$ (representing the current dual information) and $w_{k+1}$ (aggregating the past primal information) to guide our iterative progress. This perspective may suggest potential links to other first-order methods, such as momentum-based algorithms \cite{nesterov2013gradient} and Frank-Wolfe-type algorithms \cite{frank1956algorithm,jaggi2013revisiting}, though establishing an exact connection requires further investigation and is left for future work. Finally, our bundle-based Augmented Lagrangian framework also complements the viewpoints of interior-point and penalty methods \cite[Chapter 11]{boyd2004convex} in addressing $\Omega$. Our idea is to approximate the complicated convex set $\Omega$ using a line segment while interior-point and penalty methods use a suitable barrier or penalty function. \hfill $\square$
\end{remark}

\subsection{\alg{} in the iALM framework}
\label{subsection:AALM-iALM}
As highlighted in \cref{section:BALA}, \alg~is designed in the general framework of ALM. It also shares a similar algorithmic structure with the iALM in \cref{alg:alm}. We here establish a close relation between \alg{} and \cref{alg:alm}. This relationship allows us to utilize a recent result in \cite{xu2021iteration} to prove the improved sublinear rate in \cref{thm:average-iterate}. 

Our algorithm \alg{} features a single-loop process, unlike the explicit nested structure in the classical iALM. Still, the iterates $x_{k+1}$ and $y_{k+1}$ in \alg{} are updated only for the descent step, i.e., the descent test \cref{eq:test} is satisfied. In contrast, the iterates and the objective function remain the same between two consecutive descent steps, i.e.,  when \cref{eq:test} is not satisfied. Thus, we expect that the satisfaction of \cref{eq:test} means that the updates  $x_{k+1}$ and $y_{k+1}$ provide meaningful progress. Indeed, as shown in \cref{prop:connection}, each update of $x_{k+1}$ and $y_{k+1}$ implies that $x_{k+1}$ is an inexact solution to the original subproblem \eqref{eq:ALM-exact-min} with the inexactness quantified by $y_{k+1}$. For convenience, we re-iterate \cref{prop:connection} in the following result. 
\begin{lemma} \label{lemma:BALA-ALM}
    If step $k$ of \alg{} satisfies the criterion \cref{eq:test}, then  $x_{k+1}$ is an $\epsilon_{k+1}$-suboptimal solution to $\min_{x \in \Omega} \calL_{\rho}(x,y_k)$ in \cref{eq:ALM-inexact-min}, i.e., we have
    \begin{equation*}
        \calL_{\rho}(x_{k+1},y_k) - \min_{x \in \Omega} \calL_{\rho}(x,y_k)\leq \epsilon_{k+1},
    \end{equation*}
    with the suboptimality measure as 
    $
    \epsilon_{k+1} = (g(y_k) - g(y_{k+1}))/\beta. 
    $
\end{lemma}
\Cref{lemma:BALA-ALM} is just a restatement of \Cref{prop:connection}. Thus, the proof is the same as that in \Cref{prop:connection}. 
Both \alg~and iALM employ the same gradient ascent strategy for the dual update. In this sense, it is clear that our \alg{} in~\cref{alg:bundle-dual} is closely related to the iALM in \cref{alg:alm}. 
 Moreover, we can interpret that the goal of the null steps in \alg~is to search for an acceptable primal point $x_{k+1}$, which is similar to the spirit of \eqref{eq:ALM-inexact-min} in iALM. Instead of applying another iterative algorithm to solve \eqref{eq:ALM-inexact-min} aiming to satisfy \eqref{eq:subproblem-error}, \alg{} solves the sequence of problems of the form \eqref{eq:ALM-exact-set-min-main-text} with the sequence of inner approximations $\{\Omega_{k}\}$ to satisfy \eqref{eq:test}. Note that the sequence of problems of the form \eqref{eq:ALM-exact-set-min-main-text} is guided by an approximated dual function (see \Cref{eq:dual-bala-proximal}). As guaranteed by \cref{lemma:null-steps}, the number of total null steps is upper bounded by $\mathcal{O}(\epsilon^2)$. Our discussions in \Cref{appendix:computations-of-BALA} have shown that all the computations in \alg{} can be extremely efficient. With the interpretation in \cref{lemma:BALA-ALM}, many recent developments in the iALM framework (such as \cite{xu2021iteration}) might be adapted to our \alg{} algorithm. We will use this insight to prove \cref{theorem:asymptotic,thm:average-iterate}.

\section{Connection to (sub)gradient methods} \label{appendix:subgradient-methods}

This section discusses some connections between \alg{} and the dual (sub)gradient method. To start with, we review a classical result in convex optimization.
\begin{lemma}
    \label{lemma:dual-subgradient}
    Consider the function $g$ in \eqref{eq:dual}, and let $y \in \RR^m$. If $x^\star$ solves $ \min_{x \in \Omega}  \calL(x,y)$, then $\Amap x^\star -  b$ is a subgradient of $g$ at $y$, i.e., $ \Amap x^\star -  b \in \partial g(y)$.
\end{lemma}
\begin{proof}
    Let $z\in \RR^m$. It holds that 
    \begin{equation*}
        \begin{aligned}
            g(z)  = -\min_{x\in \Omega} \calL(x,z) \geq - \calL(x^\star,z)         & = -(\innerproduct{c}{x^\star}+\innerproduct{z}{b-\Amap x^\star}) \\
            & = -\innerproduct{c}{x^\star} + \innerproduct{z}{\Amap x^\star-b} \\
            & = -\innerproduct{c}{x^\star} + \innerproduct{y}{\Amap x^\star-b}  +\innerproduct{z - y}{\Amap x^\star -b} \\
            & = -(\innerproduct{c}{x^\star} + \innerproduct{y}{b-\Amap x^\star})  +\innerproduct{z - y}{\Amap x^\star-b} \\
            & = g(y) + \innerproduct{z - y}{\Amap x^\star-b},
        \end{aligned}
    \end{equation*}
    where the first inequality plugs in a feasible point $x^\star$, the fourth equality adds and subtracts the term $\innerproduct{y}{\Amap \xstar - b}$, and the last equality uses the fact that $\xstar$ solves $\min_{x\in \Omega} \calL(x,y)$.
    Since $z \in \RR^m$ is arbitrary, the above confirms that $\Amap \xstar - b $ is a subgradient of $g$ at $y$.
\end{proof}

Given a $y \in \RR^m$, \cref{lemma:dual-subgradient} confirms that if we can find a point $x^\star \in \RR^n$ such that $  g(y) = -\calL(x^\star,y) = -\min_{x\in \Omega}\calL(x,y)$ (i.e., $x^\star$ archives the dual function value in the Lagrangian function at $y$), then $\Amap x^\star -  b$ is one valid subgradient of the dual function $g$. This result is standard since the Lagrangian function $\mathcal{L}(x,y)$ is linear in $y$ \cite[Chapter 4]{ruszczynski2011nonlinear}.
Recall that we have used \cref{lemma:dual-subgradient} in the proof of \cref{prop:AALM-PBM}, especially in~\cref{eq:assumption2-PBM}.  

We have the following relationship between \alg{} and the dual subgradient method. 

\begin{lemma} \label{lemma:BALA-subgradient-method}
    If the iteration $k$ in \alg{} is a decent step, then the next dual candidate point $z_{k+1}$ in \cref{eq:dual-update-ALM} can be chosen as a standard dual subgradient update, i.e., 
    $$
    z_{k+1} = y_k - \rho h_k,
    $$
    where $\rho >0$ is the Augmented Lagrangian constant, and $h_k \in \partial g(y_k)$ is a dual subgradient at $y_k$.
\end{lemma}

\begin{proof}
     
If iteration $k$ is a decent step, we have $y_{k} = z_{k}$ and the set $\Omega_{k}$ can be constructed as a singleton $\Omega_{k} = \{v_{k}\}$ where $v_{k} \in \Omega$ satisfies $g(z_{k}) = - \calL(v_{k},z_{k})$ (see \cref{subsection:bundle-inner-approximation}). In other words, the $v_{k}$ is the point that realizes the dual function value in the Lagrangian function. As $\Omega_{k}$ contains only one element $v_{k}$, the next primal candidate solution $w_{k+1}$ obtained via
\begin{align*}
    w_{k+1} = \argmin_{x \in \Omega_k} \;\calL(x,y_k)
\end{align*}
is the same as $w_{k+1} = v_{k}$. From the construction of $v_k$ and \cref{lemma:dual-subgradient}, we see that $\Amap v_k -  b$ is a subgradient of $g$ at $y_k$. Thus, the construction of the dual candidate $$z_{k+1} = y_k + \rho (b- \Amap {v_k}) = y_k -  \rho ( \Amap {v_k} - b)$$ is just a subgradient step of $g$ at $y_k$.
\end{proof}

The key insight in the proof of \cref{lemma:BALA-subgradient-method} is that our \Cref{assump:AALM} allows the next inner approximation $\Omega_{k+1}$ only to contain one point from the dual side if the iteration $k$ is a descent step. Thus, no computation is needed to update the primal candidate $w_{k+1}$. Then, the dual candidate $z_{k+1}$ is the same as a dual subgradient step. 

Another perspective is to look at the approximated dual function $g_k$ which is defined as 
\begin{align*}
    g_k(y) = - \min_{x \in \Omega_{k}} \calL(x,y) & = - \calL(v_{k},y) \\
    & =  -\calL(v_{k},y_{k}) + \innerproduct{y_{k} - y}{b-\Amap v_{k} } \\
            & = g(y_{k})+\innerproduct{y_{k} - y}{b-\Amap v_{k}}, \\
            & = g(y_{k})+\innerproduct{\Amap v_{k} - b }{y - y_{k}}.
\end{align*} 
Since $\Amap v_{k} - b \in \partial g(y_k)$, we see that $g_k$ is the first-order linearization of $g$ at the point $y_k$. From the proximal mapping interpretation in \eqref{eq:prox-gkj}, it becomes clear that the dual candidate $z_{k+1}$ is simply obtained by a subgradient step.

\begin{remark}[Improved performance compared to naive subgradient methods] While \cref{lemma:BALA-subgradient-method} highlights a close relationship between our algorithm \alg{} and the standard dual subgradient method, \alg{} enjoys significantly improved convergence behaviors.  In particular, the standard dual subgradient method would accept $z_{k+1}$ as the next iterate, whereas \alg{} updates $z_{k+1}$ only if the reduction in the objective value meets at least a fraction of the predicted decrease; see \cref{eq:test}. Additionally, for null steps, \alg{} will also 
aggregate information from past iterations into the next inner approximation. This seemingly subtle refinement enables us to establish both asymptotic and sublinear convergence results in  \Cref{thm:average-iterate,theorem:asymptotic,theorem:convergence-bundle-dual} for any augmented Lagrangian constant $\rho$. However, subgradient methods typically require a carefully controlled diminishing step size sequence to guarantee convergence \cite[Theorem 7.4.]{ruszczynski2011nonlinear}; an improperly chosen step size may result in slow progress, oscillatory behavior, or, in some cases, failure to converge.  \hfill $\square$
\end{remark}

\section{Missing proofs in \cref{section:convergence}}
\label{apx:section:missing-proof}
In this section, we complete the missing proofs in convergence guarantees of \alg{} (\cref{section:convergence}).

\subsection{Proof of \cref{prop:connection}}
\label{apx:prop:connection}
Since $ \calL_{\rho}(x,y_k) \geq \calL(x,y_k)$ for all $x \in \RR^n$, it holds that
\begin{align*}
    \min_{x\in \Omega} \calL_{\rho}(x,y_k) & \geq \min_{x \in \Omega} \calL(x,y_k) = -g(y_k),
\end{align*}
where the last equality comes from the definition of $g$. Using the above inequality, we can quantify the progress in the primal space made by $w_{k+1}$ as
\begin{equation*}
    \begin{aligned}
         \calL_{\rho}(w_{k+1},y_k) - \min_{x \in \Omega} \calL_{\rho}(x,y_k) 
         &\leq  \calL_{\rho}(w_{k+1},y_k) + g(y_k)\\
         &= g(y_k) -g_k(z_{k+1}) - \frac{1}{2\rho }\|z_{k+1} - y_k\|^2,
    \end{aligned}
\end{equation*}
where the equality uses \cref{eq:prox-gkj}. If, further, the step $k$ is a descent step (i.e., \eqref{eq:test} holds), then $x_{k+1} = w_{k+1}$ the inexactness can be further estimated by the descent in the dual space as
\begin{align*}
    \calL_{\rho}(x_{k+1},y_k) - \min_{x \in \Omega} \calL_{\rho}(x,y_k)&\leq  g(y_k)- g_k(z_{k+1}) \\
    &\leq   \frac{1}{\beta} (g(y_k) - g(z_{k+1})).
\end{align*}
\subsection{Proof of \cref{lemma:bundle-dual-consecutive}}
\label{apx:lemma:bundle-dual-consecutive}
The proof is simple and follows from a few steps. Specifically, it holds that 
    \begin{equation*}
\begin{aligned}
     \frac{1}{2 \rho} \|z_{k+1} - y_k\|^2 & \overset{(a)}{\leq}  g_k(y_{k})  - g_k(z_{k+1}) \\ 
    &\overset{(b)}{\leq} g(y_{k})  - g_k(z_{k+1})  \\
    & \overset{(c)}{=} \frac{1}{\beta} (g(y_{k})-g(z_{k+1})) \\
    & \overset{(d)}{\leq} \frac{1}{\beta} (g(y_{k})-\gstar),
\end{aligned}
 \end{equation*}
  where $(a)$ is from the fact that $z_{k+1}$ minimizes the subproblem $\min_{y \in \RR^m} g_k(y) + \frac{1}{2\rho}\|y - y_k\|^2$ (see \cref{eq:dual-bala-proximal}), $(b)$ uses the construction $g_k(y) \leq g(y)$ for all $y \in \RR^m$, $(c)$ applies the definition of a descent step \cref{eq:test}, and $(d)$ is due to $\gstar \leq g(y)$ for all $y \in \RR^m$. The inequality \eqref{eq:bound-dual} is then proved by setting $y_{k+1} = z_{k+1}$ as a descent step happens. 
   
  Afterward, \eqref{eq:bound-affine} follows from the dual update $z_{k+1} = y_k + \rho (b-\Amap x_{k+1})$ and the above inequality as
  \begin{align*}
      \|b-\Amap x_{k+1}\|^2 =\frac{1}{\rho^2} \|z_{k+1}-y_k\|^2 \leq \frac{2}{\rho \beta } (g(y_{k})-\gstar)
  \end{align*}
  which is equivalent to \eqref{eq:bound-affine}.
  
\subsection{Proof of \cref{lemma:cost-value}}
\label{apx:lemma:cost-value}
 Using the dual update $y_{k+1} = y_k + \rho (b-\Amap x_{k+1})$ and \cref{eq:ALM-residual}, the upper bound can be established as
        \begin{equation*}
        \begin{aligned}
            \frac{1}{2\rho}  \|y_k\|^2 -\|y_{k+1}\|^2   
            &=  \rho \innerproduct{y_k}{\Amap x_{k+1} -b } + 2\rho \|\Amap x_{k+1} -b \|^2 \\
            &\leq    \|y_k\| \|y_k - y_{k+1}\| +   \frac{1}{2\rho}  \|y_k - y_{k+1}\|^2 .
        \end{aligned}
        \end{equation*}
        We move on to show the lower bound. Let $\xstar$ and $\ystar$ be any primal and dual solutions respectively. The lower bound follows from a standard saddle point argument 
        \begin{align*}
           \pstar  = \calL(\xstar   ,\ystar) & \leq  \calL(x,\ystar) \\
           & = \innerproduct{c}{x} + \innerproduct{\ystar}{b - \Amap x},  \forall x \in \Omega.
        \end{align*}
        This implies 
        \begin{equation*}
            \innerproduct{c}{x} - \pstar \geq -\|\ystar\| \|b - \Amap x\|,\; \forall x \in \Omega.
        \end{equation*}
        Plugging $x = x_{k+1}$ in the above inequality and using \cref{eq:ALM-affine} completes the proof.

\subsection{Proof of \cref{theorem:asymptotic}}
\label{apx:theorem:asymptotic}
\Cref{prop:connection} shows that a descent step implies that \alg{} has found an inexact solution to the subproblem \eqref{eq:ALM-inexact-min}. Viewing the convergence result in \cref{thm:ALM-convergence}, we see that \cref{theorem:asymptotic} will follow immediately if the accumulative inexactness is summable. We show that the accumulative inexactness is summable below.

Let $\mathcal{S} = \{ k\in \mathbb{N} :\text{step } k \text{ is a descent step}\} \cup \{1\}$ be the set of indices for descent steps. Applying \cref{prop:connection} yields
\begin{equation}
    \begin{aligned}
        \label{eq:sum-bounded}
         { \sum_{k \in \mathcal{S}}}  \calL_\rho(x_{k+1},y_k) - \min_{x\in \Omega} \calL_\rho(x,y_k)
      &\leq   \frac{1}{\beta} { \sum_{k \in \mathcal{S}}}g(y_k) - g(y_{k+1}) \\
      &\leq \frac{g(y_1) - g^\star}{\beta} \\
      & < \infty,
    \end{aligned}
\end{equation}
where the second inequality is because the summation is a telescoping sum, $g(y)$ is finite for all $y \in \RR^m$, and $g^\star \leq g(y)$ for all $y \in \RR^m$.     

\subsection{Proof of \cref{theorem:convergence-bundle-dual}}
\label{apx:theorem:convergence-bundle-dual}
As \alg{} is equivalent to a realization of the PBM in the dual (see \cref{subsection:AALA-PBM}), it boils down to show that the dual function $g$ is Lipschitz continuous and all convergence results follow from \cref{thm:convergence-bundle}.

 Let $p : \RR^m \to \RR$ as $ p(y) = \max_{x \in \Omega } \innerproduct{\Ajmap y - c}{x}$ be latter part in \cref{eq:dual}. The subdifferential of $p$ satisfies 
 $$
 \partial p(y) = \{ \Amap x : x \in \Omega, p(y) = \innerproduct{\Ajmap y - c}{x} \} \subseteq  \{ \Amap x : x \in \Omega \},\; \forall y \in \RR^m.
 $$ 
 As the set $\Omega$ is compact, the diameter $D = \max_{y,z \in \Omega} \|y-z\|$ is finite. Given a $y \in \RR^m$ and $v\in \partial p(y)$, a subgradient of $g$ can be computed as $b + v.$ Therefore, subgradients of $g$ can be bounded as 
    \begin{equation*}
         \|s\|\leq \|b\| + \|\Amap\| D, \forall s \in \partial g(y).
    \end{equation*}
    Thus, the function $g$ is $(\|b\| + \|\Amap\|D)$-Lipschitz. The function $g$ is also convex by the definition. Applying \cref{thm:convergence-bundle} directly yields the rate of $\bigO(\epsilon^{-3})$ to find an iterate $y_k$ satisfying $g(y_k) - g^\star \leq \epsilon$. Since the sequence $\{y_k\}$ is bounded by \cref{theorem:asymptotic}, applying \cref{lemma:bundle-dual-consecutive,lemma:cost-value} shows the rate of $\bigO(\epsilon^{-6})$ for $\|\Amap x_{k} - b\|$ and $|\innerproduct{c}{x_{k}} - \pstar|$.

    Finally, viewing the reciprocal relationship between the penalty parameter $\rho$ and the proximal mapping parameter in \cref{eq:dual-bala-proximal} or \eqref{eq:prox-gkj-d}, we see that setting $\rho = \epsilon$ yields the rate of $\bigO(\epsilon^{-2})$ to find an iterate $y_k$ satisfying $g(y_k) - g^\star \leq \epsilon$. Again, the rate of $\bigO(\epsilon^{-4})$ for $\|\Amap x_{k} - b\|$ and $|\innerproduct{c}{x_{k}} - \pstar|$ follows from \cref{lemma:bundle-dual-consecutive,lemma:cost-value}.

\subsection{Proof of \cref{thm:average-iterate}}
\label{apx:thm:average-iterate}
To prove \cref{thm:average-iterate}, we first introduce the following lemma for analyzing the iALM in \cite{xu2021iteration}.

\begin{lemma}[{\cite[Theorem 4]{xu2021iteration}}] 
    \label{lemma:egrotic-rate}
    Let $\{x_k,y_k\}$ be the sequence generated by \cref{alg:alm} with $y_1 = 0$ and $\calL_{\rho}(x_{k+1},y_k) - \min_{x \in \Omega} \calL_{\rho}(x,y_k) \leq \epsilon_k$ for all $k \geq 1$. It holds that
    \begin{align*}
        |\innerproduct{c}{\Bar{x}_k} - \pstar |  \leq &  \frac{1}{k \rho} \left( 2 \|\ystar \|^2 + \rho \textstyle\sum_{i = 1}^k \epsilon_i \right), \\
        \|\Amap \Bar{x}_k - b \| \leq &  \frac{1}{k \rho} \left ( \frac{1 + \|\ystar \|^2}{2} +   \rho  \textstyle\sum_{i = 1}^k \epsilon_i\right ),
    \end{align*}
    where $ \Bar{x}_k = \frac{1}{k}\sum_{i=1}^{k} x_i$ is the average iterate.
\end{lemma}

\Cref{lemma:egrotic-rate} shows that the cost value gap and the affine feasibility of the average iterate tend to zero, provided the accumulative inexactness does not blow up. Fortunately, the accumulative inexactness of \alg~is naturally bounded as shown in \cref{eq:sum-bounded}, leading to the following result.  

\begin{corollary}[Primal average iterate for \Cref{alg:bundle-dual}]
    \label{cor:primal-average} 
    At every descent step $k$ in \alg{} with $y_1 = 0$, the followings hold
    \begin{align*}
        |\innerproduct{c}{\Bar{x}_k} - \pstar | & \leq \frac{2 \|\ystar \|^2}{|\calS_k|\rho}   +  \frac{g(y_1) - g^\star }{ |\calS_k|\beta}, \\
        \|\Amap  \Bar{x}_k - b\| &  \leq \frac{1 + \|\ystar\|^2}{ 2 |\calS_k| \rho}  + \frac{g(y_1) - g^\star }{ |\calS_k| \beta},
    \end{align*}
    where $\calS_k = \{ j \in \mathbb{N} :\text{step } j \text{ is a descent step}, j \leq k \} \cup \{0\}$ is the set of indices of descent steps before $k$ and $\Bar{x}_k = \frac{1}{|\calS_k|} \sum_{j \in \calS_k } x_{j+1}$ is the average iterate.
\end{corollary}

\begin{proof}
The proof is straightforward. By \cref{prop:connection}, we know $\calL(x_{k+1},y_k) - \min_{x\in \Omega} \calL(x,y_k) \leq   \frac{1}{\beta} (g(y_k) - g(y_{k+1}))$ if step $k$ is a descent step. The desired results then follow from setting $\epsilon_k  =  \frac{1}{\beta} (g(y_k) - g(y_{k+1}))$ in \cref{lemma:egrotic-rate}.
\end{proof}

From \cref{cor:primal-average}, it is clear to see that the residuals  $|\innerproduct{c}{\Bar{x}_k} - \pstar |$ and $\|\Amap  \Bar{x}_k - b\|$ converge to zero in the rate of $\bigO(1/k)$ with respect to the descent steps. By also incorporating the number of null steps to generate a descent step in \cref{thm:convergence-bundle}, we immediately prove \cref{cor:primal-average}. Specifically, \cref{thm:convergence-bundle} shows that a descent step happens after at most $\bigO(1/\epsilon^{2})$ consecutive null steps. On the other hand, \cref{cor:primal-average} shows $|\innerproduct{c}{\Bar{x}_k} - \pstar |$ and $\|\Amap \Bar{x}_k - b \|$ will be smaller than $\epsilon$ in at most $\bigO(1/\epsilon)$ number of descent steps. Thus, the total number of iterations is at most $\bigO(1/\epsilon^{3})$. 

\subsection{Proof of \cref{thm:linear-convergence}}
\label{apx:thm:linear-convergence}
We start by showing the following important theorem, from which \cref{thm:linear-convergence} will readily follow.

\begin{theorem}
    \label{thm:error-quadratic}
     Let $\alpha > 0$, $x_k \in \RR^n$, and $f:\RR^n \to \RR$. Suppose $f_k :\RR^n \to \RR $ is a convex function that satisfy 
    \begin{equation}
    \label{eq:close-upper}
        f_k(x) \leq f(x) \leq f_k(x) + \frac{\alpha}{2}\|x - x_k\|^2, \; \forall x \in \RR^n.
    \end{equation}
    If $\theta \geq \alpha$ and $x_{k+1} = \argmin_{x} f_k(x) + \frac{\theta}{2}\|x - x_k\|^2$, then it holds that
    \begin{align*}
    0 \leq f(x_{k+1}) - f_k(x_{k+1}) \leq  f(x_k)  - f(x_{k+1}).
\end{align*}
Furthermore, suppose $f$ is convex and there exists a constant $\mu > 0$ such that  
\begin{equation}
    \label{eq:QG-f}
    \frac{\mu}{2} \cdot \Dist^2(x,S) \leq f(x) - f^\star,\; \forall x \in \RR^n,
\end{equation}
where $f^\star = \min_{x\in \RR^n} f(x)$ and $S = \argmin_{x\in \RR^n} f(x)$ is assumed to be nonempty. Then it holds that
\begin{align*}
    \Dist(x_{k+1},S) \leq \sqrt{\frac{\theta}{\theta +\mu}}\Dist(x_k,S).
\end{align*}
\end{theorem}
\begin{proof}
    We note that the function $f_k + \frac{\theta}{2}\|\cdot-x_k\|$ is $\theta$-strongly convex. Using the strongly convex first-order inequality for $f_k + \frac{\theta}{2}\|\cdot-x_k\|$ at $x_{k+1}$ with $0 \in \partial (f_k(x_{k+1}) + \frac{\theta}{2}\|x_{k+1}-x_k\|^2) $, we see that 
\begin{align*}
     f_{k}(x_{k+1}) + \theta \|x_{k+1} - y_k\|^2& = f_{k}(x_{k+1}) + \frac{\theta}{2} \|x_{k+1} - x_k\|^2 + \frac{\theta}{2} \|x_k - x_{k+1}\|^2 \\
     & \leq f_k(x_k) \\
     & \leq f(x_k),
\end{align*}
where the last inequality uses $f_{k} \leq f$. Using \eqref{eq:close-upper} with $x = x_{k+1}$ and the above inequality, we see that 
\begin{align*}
    f(x_k) - f(x_{k+1})  &\geq f_{k}(x_{k+1}) + \theta \|x_{k+1} - y_k\|^2 - f_k(x_{k+1})  - \frac{\alpha}{2}\|x_{k+1} - x_k\|^2 \\
    & \geq \frac{\alpha}{2} \|x_{k+1} - y_k\|^2,
\end{align*}
where the last inequality uses $\theta \geq \alpha$. Therefore, we have 
\begin{align*}
    f(x_{k+1}) - f_k(x_{k+1}) & \leq \frac{\alpha}{2}\|x_{k+1} - x_k\|^2 \\
    & \leq f(x_k) - f(x_{k+1}),
\end{align*}
where the first inequality comes from \eqref{eq:close-upper}.

We now move on to prove the second result. Again, using the first-order inequality for the strongly convex function $f_k + \frac{\theta}{2} \|\cdot - x_k\|^2$ at $x_{k+1}$ with $0 \in \partial ( f_k(x_{k+1}) + \frac{\theta}{2} \|x_{k+1} - x_k\|^2)$, we see that 
\begin{align*}
    f_{k}(x_{k+1}) + \frac{\theta}{2} \|x_{k+1} - x_k\|^2 + \frac{\theta}{2} \|\Pi_{S}(x_k)  - x_{k+1}\|^2 & \leq f_k(\Pi_{S}(x_k)) + \frac{\theta}{2}\|\Pi_{S}(x_k) - x_k\|^2 \\
    &\leq f^\star + \frac{\theta}{2}\Dist^2(x_k,S),
\end{align*}
where $\Pi_{S}(x_k) = \argmin_{x \in S} \|x-x_k\|$ is the projection of $x_k$ on to the set $S$.
Applying \eqref{eq:close-upper} in the above inequality yields 
\begin{equation*}
    f(x_{k+1})+ \frac{\theta}{2} \|\Pi_{S}(x_k)  - x_{k+1}\|^2 \leq  f^\star + \frac{\theta}{2}\Dist^2(x_k,S).
\end{equation*}
Finally, applying \eqref{eq:QG-f} to above inequality and using the fact $\|\Pi_{S}(x_k)  - x_{k+1}\|  \geq \Dist(x_{k+1},S)$, we arrive at 
\begin{align*}
    & \frac{\mu}{2}\Dist^2(x_{k+1},S)  + \frac{\theta}{2}\Dist^2(x_{k+1},S)  \leq \frac{\theta}{2}\Dist^2(x_k,S) \\
   \Longrightarrow \quad & \Dist(x_{k+1},S) \leq \sqrt{\frac{\theta}{\theta +\mu}}\Dist(x_k,S).
\end{align*}
This completes the proof. 
\end{proof}
Aside from the linear reduction, one important message from \cref{thm:error-quadratic} is that the proximal point $x_{k+1}$ is guaranteed to be as good as the previous point $x_k$ even though the function $f_k$ is not the same as the true function $f$, provided the proximal parameter is large enough. It is this observation that allows us to show that a descent step will always happen.

With \cref{thm:error-quadratic}, we are ready to show that every step in \alg{} is a descent step under the assumptions in \cref{thm:linear-convergence}. Recall \cref{eq:dual-bala-proximal} shows that the dual candidate $z_{k+1}$ can be equivalently acquired as $$z_{k+1} = \argmin_{y \in \RR^m} g_{k}(y) + \frac{1}{2\rho}\|y- y_k\|^2.$$

Since we assume that $g_k$ satisfies \eqref{eq:quadratic-close} and $\frac{1}{\rho} \geq \gamma$ where $\gamma$ is the constant in \eqref{eq:quadratic-close}, applying \cref{thm:error-quadratic} on $g$ and $g_k$ shows
\begin{align*}
    \beta (g(y_k) - g_k(z_{k+1})) & = \beta (g(y_k) -g(z_{k+1})) + \beta (g(z_{k+1}) - g_k(z_{k+1}))  \\
    & \leq \beta (g(y_k) -g(z_{k+1})) + \beta( g(y_k)  - g(z_{k+1})) \\
    & = 2 \beta (g(y_k) -g(z_{k+1}))\\
    & \leq  g(y_k) -g(z_{k+1}),
\end{align*}
where the last inequality uses the choice $\beta \in (0,1/2]$ and $g(y_k) - g(z_{k+1}) \geq 0$. 
Hence, the test criterion \eqref{eq:test} is satisfied and we conclude that every step $k \geq T$ is a descent step.

Moreover, \cref{thm:error-quadratic} also shows that 
\begin{align*}
   \Dist(y_{k+1},\dsolution) = \Dist(z_{k+1},\dsolution) \leq \sqrt{\frac{1/\rho}{1/\rho +\alpha}}\Dist(y_k,\dsolution),
\end{align*}
where $\dsolution$ is the optimal solution set of the dual problem \eqref{eq:dual}, $\alpha$ is the quadratic growth constant in \eqref{eq:QG}, and the equality is due to a descent step. We see that the dual sequence $\{y_k\}$ converges linearly to the solution set. Since the dual function $g$ is $(\|b\| + \|\Amap\|D)$-Lipschitz as shown in the \cref{apx:theorem:convergence-bundle-dual}, where $D$ is the diameter of the set $\Omega$, the cost value gap $g(y_k) - g^\star$ also converges linearly as 
\begin{align}
    g(y_k) - g^\star \leq(\|b\| + \|\Amap\|D) \Dist(y_k,\dsolution). \label{eq:linear-reduction-g}
\end{align}

Finally, the convergence on the primal iteration follows from \cref{lemma:bundle-dual-consecutive,lemma:cost-value} as 
\begin{align}
     \| \Amap x_{k+1} - b\|^2  \leq &\frac{2}{\beta \rho } (g(y_{k}) - \gstar),  \label{eq:Ax-1}\\
 -\|\ystar\|\|\Amap x_{k+1} -  b\|   \leq \innerproduct{c}{x_{k+1}} - p^\star  &\leq  \frac{2}{\beta}(g(y_k)  -  \gstar)  + \rho^2 \|y_k\|  \|\Amap x_{k+1} - b\|. \label{eq:cx-p}
\end{align}
From \eqref{eq:Ax-1} and \eqref{eq:linear-reduction-g}, we have 
\begin{align*}
    \| \Amap x_{k+1} - b\|^2 &  \leq \frac{2}{\beta \rho }(\|b\| + \|\Amap\|D) \Dist(y_k,\dsolution).
\end{align*}
From \eqref{eq:cx-p} and \eqref{eq:linear-reduction-g}, we also have 
\begin{align*}
    |\innerproduct{c}{x_{k+1}} - p^\star|^2 & \leq \max\{ \underbrace{\|\ystar\|^2 \| \Amap x_{k+1} - b\|^2}_{R_1},\underbrace{(\frac{2}{\beta}(g(y_k)  -  \gstar)  + \rho^2 \|y_k\|  \|\Amap x_{k+1} - b\|)^2}_{R_2} \}.
\end{align*}
The first term $R_1$ can be bounded as 
\begin{equation*}
    R_1 \leq \|\ystar\|^2 \frac{2}{\beta \rho }(\|b\| + \|\Amap\|D) \Dist(y_k,\dsolution).
\end{equation*}
The second term $R_2$ can also be bounded as 
\begin{align*}
    R_2 & \overset{(a)}{\leq}  \frac{4}{\beta^2}(g(y_k)  -  \gstar)^2 + \frac{4}{\beta}(g(y_k)  -  \gstar) \rho^2 \|y_k\|  \|\Amap x_{k+1} - b\| + \rho^4 \|y_k\|^2  \|\Amap x_{k+1} - b\|^2 \\
    & \overset{(b)}{\leq}  \frac{4}{\beta^2}(g(y_k)  -  \gstar)^2 + \frac{4}{\beta}(g(y_k)  -  \gstar) \rho^2 \|y_k\|  \sqrt{\frac{2}{\beta \rho } (g(y_{k}) - \gstar)} + \rho^4 \|y_k\|^2  \frac{2}{\beta \rho } (g(y_{k}) - \gstar) \\
    & \overset{(c)}{=}  \left(\frac{4}{\beta^2}(g(y_k)  -  \gstar) + \frac{4}{\beta} \rho^2 \|y_k\|  \sqrt{\frac{2}{\beta \rho } (g(y_{k}) - \gstar)} + \rho^4 \|y_k\|^2  \frac{2}{\beta \rho }\right )  (g(y_{k}) - \gstar)\\
    & \overset{(d)}{\leq}  \left(\frac{4}{\beta^2}(g(y_T)  -  \gstar) + \frac{4}{\beta} \rho^2 E  \sqrt{\frac{2}{\beta \rho } (g(y_{T}) - \gstar)} + \rho^4 E^2  \frac{2}{\beta \rho }\right )  (\|b\| + \|\Amap\|D) \Dist(y_k,\dsolution),
\end{align*}
where $(a)$ comes from squaring both sides of \eqref{eq:cx-p}, $(b)$ applies \eqref{eq:Ax-1}, $(c)$ factors out the common term $(g(y_{k}) - \gstar)$, and $(d)$ uses the fact that $g(y_{k}) - \gstar \leq g(y_{T}) - \gstar$ for all $k \geq T $ and sets $E$ as constant such that $\|y_k\| \leq E$ for all $k \geq 1$. We note that the constant $E$ exists as the sequence $\{y_k\}$ is convergent \cref{theorem:asymptotic}. 

Thus, the two constants $\mu_1$ and $\mu_2$ such that $\mu_1 \in (0,1), \mu_2 > 0$, 
\begin{align*}
    \Dist(y_{k+1}, \dsolution) & \leq \mu_1 \cdot \Dist(y_{k}, \dsolution), \\
    \max\{g(y_k) - g^\star, \|\Amap x_{k} - b \|^2, | \innerproduct{c}{x_k} - \pstar |^2\} &\leq \mu_2  \cdot  \Dist(y_{k}, \dsolution),
\end{align*}
can be chosen as $\mu_1   = \sqrt{\frac{1/\rho}{1/\rho +\alpha}}$ and 
\begin{align*}
      \mu_2 = \max\left\{ 1,\frac{2 \|\ystar\|^2}{\beta \rho } , \left(\frac{4}{\beta^2}(g(y_T)  -  \gstar) + \frac{4}{\beta} \rho^2 E  \sqrt{\frac{2}{\beta \rho } (g(y_{T}) - \gstar)} + \rho^4 E^2  \frac{2}{\beta \rho }\right )  \right\} (\|b\| + \|\Amap\|D).
\end{align*}
This completes the proof of \cref{thm:linear-convergence}.

\section{Illustration of \alg{} with a simple two-dimensional LP instance}
\label{apx:section:illustration}
In this section, we use a simple two-dimensional example to demonstrate the process of \alg{}. We consider the simple primal and dual linear program

\begin{minipage}{0.49\textwidth}
\begin{equation*}
       \begin{aligned}
        \min_{x \in \RR^2} & \quad x_1 + x_2 \\\mathrm{subject~to} & \quad  2x_1 + x_2 = 1, \\
        & \quad x \in \RR^2_+, |x|_1 \leq 1,
    \end{aligned}
\end{equation*}
\end{minipage}
\begin{minipage}{0.49\textwidth}
\begin{equation*}
    \begin{aligned}
        -\min_{y \in \RR} & \quad g(y) = -y + \max\{ 2y - 1, y-1 , 0\}.
    \end{aligned}
\end{equation*}
\end{minipage}
The above primal problem is as the form of \eqref{eq:primal} with the problem data $c = \begin{bmatrix}
    1,1
\end{bmatrix}^\tr,\Amap = \begin{bmatrix}
    2, 1
\end{bmatrix}, b = 1$, and $\Omega = \{x\in \RR^2_+ : |x|_1 \leq 1  \}$. The primal problem has the unique optimal solution $\xstar = (0.5,0)$. 
We run \alg{} with the initializations
$$x_1 = w_1 = (0.5,0.5), y_1 = 0,v_1 = (0,0), \beta = 0.25,~\text{and}~\rho = 1.5,$$
and consider two different constructions of the inner approximation: 1) line segment approximation and 2) convex hull approximation with three points. The convex hull approximation with three points always contains the line approximation. In both constructions, the subproblem admits an analytical solution.

As discussed in \cref{apx:subsection:compuation}, at iteration $k$, the dual information $v_{k+1}$ is constructed as one of $(0,0),(1,0),$ or $(0,1)$. In this example, we have the explicit form 
\begin{align*}
    v_{k+1} = \begin{cases}
        (0,0),  & \text{if } y > \frac{1}{2}, \\ 
        (1,0),  & \text{if } y \leq  \frac{1}{2},
    \end{cases}
\end{align*}
since $2y - 1 \geq  y-1 ,\forall y \geq 0$. 

\vspace{3pt}

\noindent \textbf{Line segment approximation.} 
In \cref{fig:visualization-line,fig:convergence-line}, we present a visualization of \alg{} and its convergence behavior.  
At the first iteration, the constraint set \(\Omega_1\) is defined as the line segment between \(x_1\) and \((0,0)\) (blue line in \cref{fig:visualization-line}). Solving the subproblem yields the primal candidate point \(w_2\) (orange line in \cref{fig:visualization-line}). From the candidate \(w_2\), we construct the next inner approximation as the line segment between \(w_2\) and \((1,0)\). This process continues iteratively: generating a new candidate point \(w_3\), forming \(\Omega_3\) as the line segment between \(w_3\) and \((0,0)\), and repeating the update.  

We observe that the candidate solution \(w_k\) always remains within the set \(\Omega\), though the affine constraint may be violated. Readers familiar with Frank-Wolfe algorithms might initially see \alg{} as equivalent to a Frank-Wolfe algorithm with line search. However, \alg{} differs in key aspects. In a Frank-Wolfe algorithm, all iterates strictly remain within the feasible region, ensuring constraint satisfaction at every step. In contrast, \alg{} permits infeasibility in the affine constraint. Moreover, \alg{} incorporates an additional testing step \eqref{eq:test}, which serves as a safeguard to determine whether the current iterate should be accepted. Unlike \alg{}, the Frank-Wolfe algorithm always accepts new iterates, without assessing their quality before updating.

\vspace{3pt}

\noindent \textbf{Convex hull with three points.} 
Despite the simplicity and the convergence guarantee of using the line segment approximation, the convergence can be slow when the solution is on the boundary. In this case, it is more effective to use a stronger approximation to better capture the solution. At iteration $k$, we consider the approximation
\begin{equation*}
    \Omega_{k} = \mathrm{conv}(0,v_k,w_k) = \{\alpha v_k + \gamma w_k : \alpha \geq 0, \gamma \geq 0 , \alpha + \gamma \leq 1 \},
\end{equation*}
which contains more information than the line segment approximation. In \cref{fig:visualization-convexhull}, we see that the first inner approximation is still a line segment between $x_1$ and $(0,0)$ as the dual information $v_1$ is $(0,0)$. Therefore, it yields the same candidate $w_2$ as the case in the line segment approximation. Starting from the second iteration, the dual information $v_k$ becomes nonzero and the constraint set $\Omega_k$ becomes a triangle. As a better approximation is considered, it is expected that the algorithm can achieve a faster convergence rate. Specifically, the algorithm obtains the candidate $w_3$ that is already closer to the optimal solution $\xstar$, compared with the candidates acquired using the line segment approximation. The better convergence performance can also be confirmed from \cref{fig:convergence-convexhull}, which shows the linear convergence of \alg{}. 

\begin{figure}[t]
     \centering
     \subfigure[Line segment update]{
{\includegraphics[width=0.22\textwidth]
{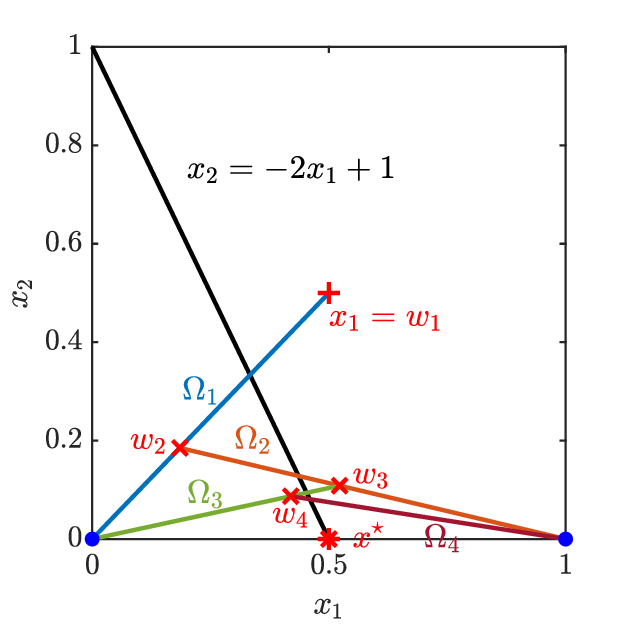}}
\label{fig:visualization-line}
}
\subfigure[Convergence behavior]{
{\includegraphics[width=0.22\textwidth]
{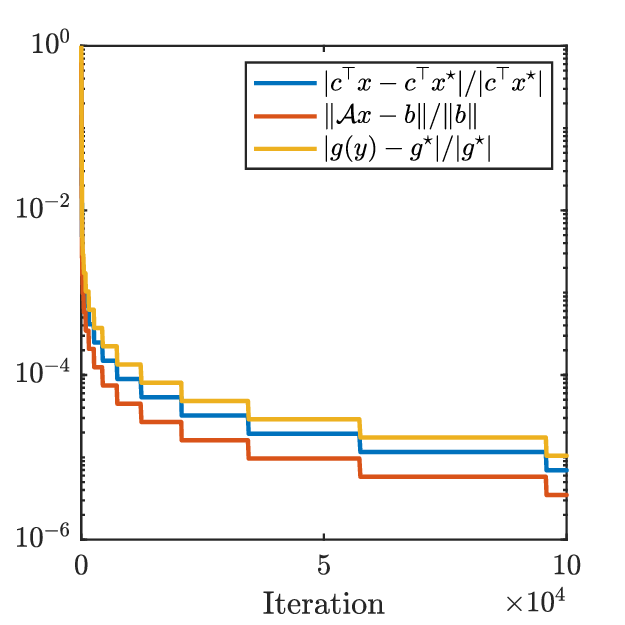}}
\label{fig:convergence-line}
}
     \subfigure[Convex hull:\! three points]{
{\includegraphics[width=0.22\textwidth]
{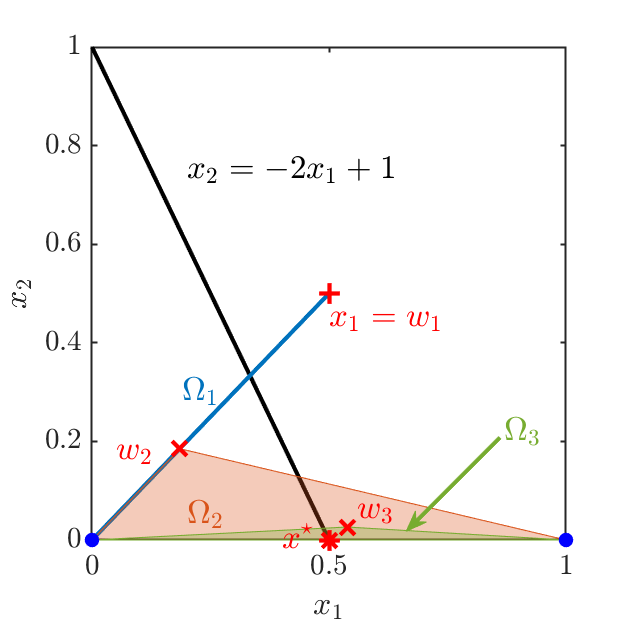}}
\label{fig:visualization-convexhull}
}
\subfigure[Convergence behavior]{
{\includegraphics[width=0.22\textwidth]
{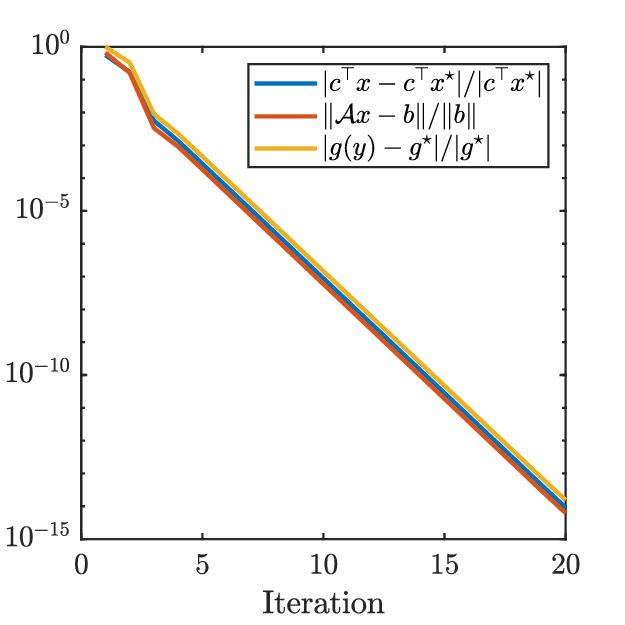}}
\label{fig:convergence-convexhull}
}
\caption{Numerical experiment for a simple linear program}
\label{fig:Illustration}
\end{figure}

\section{Application to Semidefinite Programs} 
\label{section-application-SDP}
In this section, we provide details on the application of \alg{} to solving semidefinite programs (SDPs). SDPs have a wide range of applications across various fields. Rather than reviewing these in depth (see \cite{wolkowicz2012handbook}), we simply emphasize that SDPs form a broad and powerful class of conic programs, encompassing linear programs (LPs) and second-order cone programs (SOCPs) as special cases. Therefore, all the discussions below can be easily adapted to solve any LPs and SOCPs, as well as conic programs involving a mixture of LP, SOCP, and positive semidefinite constraints.  

\subsection{SDPs with a trace constraint} \label{appendix:trace-SDP}
The standard primal and dual SDPs are optimization problems of the form 

\begin{minipage}{0.44\textwidth}
\begin{equation}
    \label{eq:primal-sdp}
    \begin{aligned}
        \min_{X \in \sn } & \quad \innerproduct{C}{X} \\\mathrm{subject~to} & \quad \Amap X = b, \\
        & \quad X \in \snp,
    \end{aligned}
\end{equation}
\end{minipage}
\begin{minipage}{0.49\textwidth}
\begin{equation}
    \label{eq:dual-sdp}
    \begin{aligned}
        \max_{y \in \RR^m} & \quad \innerproduct{b}{y} \\\mathrm{subject~to} & \quad  C - \Ajmap y  \in \snp.
    \end{aligned}
\end{equation}
\end{minipage}

\noindent where  $C \in \sn, b \in \RR^m,$ and the linear map $\Amap: \sn \to \RR^m $ are the problem data. As well see below, applying \alg{} to solve \cref{eq:primal-sdp,eq:dual-sdp} closely resembles the classical {Spectral Bundle Method (\SBM)}, originally introduced in \cite{helmberg2000spectral}; also see recent discussions in \cite{ding2023revisiting,liao2023overview}.

We assume the set of the optimal solutions to \cref{eq:primal-sdp} is compact, then there exists a constant $\gamma \geq 0$ such that $\Trace(\Xstar) \leq \gamma$ for all optimal solution $\Xstar$ in \cref{eq:primal-sdp}. This assumption is mild and less restrictive than that \cite{helmberg2000spectral}, where the equality trace constraint $\Trace(X) = \gamma$ is needed.  
Therefore, adding the constraint $\Trace(X) \leq \gamma$ into \cref{eq:primal-sdp} does not affect the optimal solution set. Let $\Omega = \{X\in \snp : \Trace(X) \leq \gamma\}$. 
We can now consider the following primal and dual problems

\begin{minipage}{0.34\textwidth}
\begin{equation}
    \label{eq:primal-sdp-bounded}
    \begin{aligned}
        \min_{X \in \sn } & \quad \innerproduct{C}{X} \\\mathrm{subject~to} & \quad \Amap X = b, \\
        & \quad X \in \Omega,
    \end{aligned}
\end{equation}
\end{minipage}
\begin{minipage}{0.61\textwidth}
\begin{equation}
    \label{eq:dual-sdp-bounded}
    \begin{aligned}
        -\min_{y \in \RR^m} \; g(y) = & 
        -\innerproduct{b}{y} + \gamma \max\{\lambda_{\max}(\Ajmap y - C),0\}.
    \end{aligned}
\end{equation}
\end{minipage}

It is clear that \cref{eq:primal-sdp-bounded} is of the form of our problem \cref{eq:primal} as the constraint set $\Omega$ becomes compact. 
The fact that \cref{eq:dual-sdp-bounded} is the dual problem of \cref{eq:primal-sdp-bounded} can be seen from the identity 
\begin{equation*}
\begin{aligned}
    \gamma \max\{\lambda_{\max}(\Ajmap y - C),0\} &=   \max\{ \gamma \lambda_{\max}(\Ajmap y - C),0\}  \\
    &= \max_{X \in \Omega} \innerproduct{\Ajmap y - C}{X},\; \forall y \in \RR^m.
\end{aligned}
\end{equation*}
Further details on the trace constraint in primal and dual SDPs can be found in \cite{liao2023overview}. 

\subsection{Applications of \alg{} in solving SDPs}

We can now directly apply \alg{} to solve \cref{eq:primal-sdp-bounded}. At iteration $k$, we have three steps: a) computing a pair of primal and dual candidates; b) testing the descent progress; c) updating the inner approximation. 

\vspace{3pt}
\noindent \textbf{Step a): primal and dual candidate.} We perform the following two steps to acquire a pair of primal and dual candidates $(W_{k+1},y_{k+1})$ 
\begin{equation}
    \label{eq:SBM-dual-subproblem}
    \begin{aligned}
        W_{k+1} & \in \argmin_{X\in \Omega_k} \; \innerproduct{C}{X} + \innerproduct{y_k}{b-\Amap X} + \frac{\rho}{2}\|b - \Amap X\|^2  \\
        z_{k+1}       & = y_k + \rho (b - \Amap W_{k+1}),
    \end{aligned}
\end{equation}
where the inner approximation set $\Omega_{k}$ is parametrized by two matrices $\Bar{X}_k \in \snp$ and $V_k \in \RR^{n \times r}$ with $\Trace(\Bar{X}_k) = 1$ and $V_k^\tr V_k = I_{r}$, where $r < n$ is a predefined number: 
$$
\Omega_k := \{\eta \Bar{X}_k + V_k S V_k^\tr : \eta \geq 0, S \in \mathbb{S}_+^r, \eta + \Trace(S) \leq \gamma\}.
$$
It is clear that $\Omega_k$ is convex and an inner approximation of $\Omega$. Similar to \cref{eq:approximated-dual}, we define the induced approximation dual function 
\begin{equation}
\label{eq:approximated-dual-sdp}
g_k(y) = - \min_{x\in\Omega_k} \innerproduct{C}{X} + \innerproduct{y}{b-\Amap X }. 
\end{equation}

\vspace{3pt}
\noindent \textbf{Step b). Testing the descent progress.} To proceed, we use the same test \cref{eq:test} to decide whether setting $X_{k+1} = W_{k+1}$ and $y_{k+1} = z_{k+1}$ (\textit{descent step}) or $X_{k+1} = X_{k}$ and $y_{k+1} = y_{k}$ (\textit{null step}). If a descent step occurs at step $k$, \cref{prop:connection} confirms that $W_{k+1}$ is an inexact solution to the problem $\min_{X \in \Omega} \mathcal{L}_{\rho}(X,y_k)$. 

\vspace{3pt}
\noindent \textbf{Step c): Updating the inner approximation.} Similar to the spectral bundle method in~\cite{helmberg2000spectral,ding2020spectral,liao2023overview}, we construct the next inner approximation  $\Omega_{k+1}$ by updating the matrices $\Bar{X}_k$ and $V_k$ to $\Bar{X}_{k+1}$ and $V_{k+1}$ respectively. 
Specifically, the dimension $r$ is decomposed as $r = r_{\mathrm{p}} + r_{\mathrm{c}}$ with $r_{\mathrm{p}} \geq 0$ and $r_{\mathrm{c}} \geq 1$. Let one optimal solution to \cref{eq:SBM-dual-subproblem} be 
\begin{subequations}
\begin{equation} \label{eq:SDP-Omega-1}
W_{k+1} = \eta^\star_k \Bar{X}_k + V_k S^\star_k V_k^\tr     
\end{equation}
and write 
\begin{equation}\label{eq:SDP-Omega-2}
S^\star_k = Q_1 \Lambda_2 Q_1^\tr +  Q_2 \Lambda_2 Q_2^\tr    
\end{equation}
as the eigenvalue decomposition, where $\Lambda_1$ consists the largest $r_\mathrm{p}$ eigenvalues and $\Lambda_2 $ consists the rest $r_\mathrm{c} = r - r_{\mathrm{p}}$ of the eigenvalues. The new matrix $\Bar{X}_{k+1}$ is set as 
\begin{equation} \label{eq:SDP-Omega-3}
    \Bar{X}_{k+1} = \frac{1}{\eta^\star_k + \Trace(\Lambda_2)} (\eta^\star_k \Bar{X}_{k} + V_k Q_2 \Lambda_2 Q_2^\tr V_k^\tr),
\end{equation}
which satisfies $\Trace(\Bar{X}_{k+1}) = 1$. On the other hand, the new orthonormal matrix $V_{k+1}$ is set to span the top $r_{\mathrm{c}}$ orthonormal eigenvectors of $\Ajmap z_{k+1} - C$ and $V_kQ_1$, i.e., 
\begin{equation} \label{eq:SDP-Omega-4}
V_{k+1} = \mathsf{orth}(\begin{bmatrix}
    v_1, \ldots, v_{r_{\mathrm{c}}},V_kQ_1
\end{bmatrix}),
\end{equation}
where $\mathsf{orth}$ denotes an orthonormalization procedure and $v_{1},\ldots,v_{r_{\mathrm{c}}}$ are the top $r_{\mathrm{c}}$ eigenvectors of $\Ajmap z_{k+1} - C$. Note that the number $r_{\mathrm{p}}$ is allowed to be zero while $r_{\mathrm{c}}$ is required to be at least greater than one.     
\end{subequations}

These update of $\bar{X}_{k+1}$ and $V_{k+1}$ can be verified to satisfy \cref{assump:AALM}. 

\begin{theorem}
\label{prop:sbm-satisfied}
The update of the inner approximation $\Omega_k \subseteq \Omega$ from \cref{eq:SDP-Omega-1,eq:SDP-Omega-2,eq:SDP-Omega-3,eq:SDP-Omega-4}~satisfies all requirements in \cref{assump:AALM}. 
\end{theorem}

The proof of \cref{prop:sbm-satisfied} requires some matrix analysis techniques, and we provide the details in \Cref{appendix:SDP-proof}.  
The above constructions are indeed the same as \SBM{}. In this sense, \SBM{} can be viewed as one algorithm in our bundle-based Augmented Lagrangian framework, and all the convergence guarantees in \cref{theorem:convergence-bundle-dual,thm:average-iterate,thm:linear-convergence} hold true when solving SDPs.

\begin{remark}[Linear convergences in solving SDPs]
    In the context of SDPs, the quadratic growth property has been known to hold locally around the solution set \cite{cui2016asymptotic,cui2017quadratic,ding2023revisiting,liao2024inexact}. Building upon the quadratic growth condition, iALM-based algorithms \cref{alg:alm} for SDPs have been shown to enjoy linear convergence \cite{cui2019r,yang2015sdpnal+}. Moreover, it is also shown that semi-smooth Netwon-CG methods are effective in solving the subproblem in iALM \cite{yang2015sdpnal+}. However, the semi-smooth Newton-CG method requires the computation of the second-order information of the Lagrangian function, which can be complicated and computationally challenging. On the other hand, recently, the \SBM{} has been revisited and shown to enjoy linear convergence \cite{ding2023revisiting,liao2023overview}. The key observation lies in the fact that the approximation function $g_k$ in \eqref{eq:approximated-dual-sdp} captures the true function $g$ in \eqref{eq:dual-sdp-bounded} up to a quadratic error, i.e., there exists a $\mu > 0$ such that  
    \begin{align*}
        g_k(y)\leq g(y) \leq g_k(y) + \frac{\mu}{2} \|y -y_k\|^2,
    \end{align*}
    provided the constraint set $\Omega_k$ in $g_k$ captures the rank of the primal solutions in \cref{eq:primal-sdp} and strict complementarity holds \cite{ding2023revisiting}. This suggests that \SBM{} is particularly effective when the primal SDP \cref{eq:primal-sdp} admits low-rank solutions as the subproblem \eqref{eq:SBM-dual-subproblem} can be solved efficiently. We also provide some numerical evidence in \cref{fig:numerical-experiment}. \hfill $\square$
\end{remark}

\subsection{Proof of \Cref{prop:sbm-satisfied}} \label{appendix:SDP-proof}

    Firstly, recall that we have $\Omega_k = \{\eta \Bar{X}_k + V_k S V_k^\tr : \eta \geq 0, S \in \mathbb{S}_+^r, \eta + \Trace(S) \leq \gamma\}$ and $V_k$ is orthonormal, so clearly $\Omega_k  \subseteq \Omega = \{X \in \snp : \Trace(X) \leq \gamma\}$.    
    Secondly, we argue that we can find $v_{k+1} \in \Omega_{k+1}$ such that 
        \begin{align*}
            \calL(v_{k+1},z_{k+1}) & = -g(z_{k+1})  \\
            & = -\innerproduct{b}{z_{k+1}} + \gamma \max\{\lambda_{\max}(\Ajmap z_{k+1}-C),0\}.
        \end{align*}
        Suppose $\lambda_{\max}(\Ajmap z_{k+1}-C) < 0$. Then choosing $v_{k+1} = 0 \in \Omega_{k+1}$, we see that
        \begin{equation*}
            \begin{aligned}
                \calL(v_{k+1},z_{k+1})
                 &= \innerproduct{b}{z_{k+1}}  \\
                 &= \innerproduct{b}{z_{k+1}} - \gamma \max\{\lambda_{\max}(\Ajmap z_{k+1}-C),0\} \\
                 &= -g(z_{k+1}).
            \end{aligned}
        \end{equation*}
        On the other hand, suppose $\lambda_{\max}(\Ajmap z_{k+1}-C) \geq 0$. Let $v_+ \in \RR^n$ such that $v_+^\tr(\Ajmap z_{k+1}-C)v_+ = \lambda_{\max}(\Ajmap z_{k+1}-C)$. As $v_+  \in \mathrm{Span}(V_{k+1})$ by the design of $V_{k+1}$, we know there exists a unit vector $e \in \RR^r$ such that $V_{k+1}e = v_+$. Let $\eta = 0$, $S = \gamma e e^\tr$, and  $$v_{k+1}  = \eta \bar{X}_{k+1} + V_{k+1} S V_{k+1}^\tr = \gamma V_{k+1}  e e^\tr V_{k+1}^\tr = \gamma v_+ v_+^\tr.$$ We see that $  v_{k+1} \in \Omega_{k+1}$ and 
            \begin{equation*}
            \begin{aligned}
            \calL (v_{k+1},z_{k+1})& = \innerproduct{C}{v_{k+1}} + \innerproduct{z_{k+1}}{b-\Amap v_{k+1} } \\
                & = \innerproduct{C}{\gamma v_+ v_+^\tr} + \innerproduct{z_{k+1}}{b-\Amap (\gamma v_+ v_+^\tr)} \\
                & = \innerproduct{b}{z_{k+1}} + \innerproduct{C}{\gamma v_+ v_+^\tr} - \innerproduct{z_{k+1}}{\Amap(\gamma v_+ v_+^\tr)} \\
                & = \innerproduct{b}{z_{k+1}} + \innerproduct{C- \Ajmap z_{k+1}}{\gamma v_+ v_+^\tr} \\
                & = \innerproduct{b}{z_{k+1}} - \gamma \innerproduct{ \Ajmap z_{k+1} -C}{ v_+ v_+^\tr} \\
                & = \innerproduct{b}{z_{k+1}} - \gamma \lambda_{\max}(\Ajmap z_{k+1}-C).
            \end{aligned}
        \end{equation*}    
Lastly, by the design of $V_{k+1}$, there exists a $l \in \RR^{r \times r_{\mathrm{p}}}$  such that $V_{k+1}l=V_kQ_1$ and $l^\tr l = I_{r_{\mathrm{p}}}\in \mathbb{S}^{r_\mathrm{p}}$. Choosing $\eta = \eta^\star_k + \Trace(\Lambda_2)$ and $S = l \Lambda_1 l^\tr$, we see that $$\eta \Bar{X}_{k+1} + V_{k+1} S V_{k+1} = \eta^\star_k \bar{X}_{k} + V_k Q_2 \Lambda_2 Q_2^\tr V_k^\tr + V_k Q_1 \Lambda_1 Q_1^\tr V_k^\tr =W_{k+1} \in \Omega_{k+1},$$
where the first equality applies \eqref{eq:SDP-Omega-3}.

\section{Details in \cref{section:experiment} and additional numerical experiments}
\label{apx:section:detail-numerical}
This section discusses how we generate the SDPs in \cref{section:experiment} in the form 
\begin{align*}
    \min_{X \in \sn} \{ \Trace(CX) : \Amap X = b, \Trace(X) \leq a ,X \in \snp\},
\end{align*}
and present additional experiments in \cref{fig:numerical-experiment-more-1}. All numerical experiments are conducted on a PC with a 12-core Intel i7-12700K CPU@3.61GHz and 32GB RAM. The algorithm is implemented in Matlab 2022b. 
\subsection{Linear convergences for randomly generated SDPs}
\label{apx:subsection-randomly}
Our first experiment considers the randomly generated SDP with dimension $n = m.$ To construct the problem data, We first generate the constraint matrices $A_1, \ldots, A_m \in \sn$ with each entry generated by a normal distribution with mean $0$ and variance $1$.
To further guarantee the primal problem satisfies Slater's condition, we let the diagonal elements of $A_1, \ldots,$ and $ A_m$ be zeros. The linear map $\Amap : \sn \to \RR^m$ is then defined as 
\begin{align*}
    \Amap X = \begin{bmatrix}
        \innerproduct{A_1}{X},\ldots,\innerproduct{A_m}{X}
    \end{bmatrix}^\tr.
\end{align*}
To generate an optimal solution, we construct a positive definite matrix \( S \in \mathcal{S}_+^n \) with the eigenvalue decomposition
\[
S =\lambda_1  v_1  v_1^\tr + \ldots + \lambda_n v_n  v_n^\tr.
\]

We then set \( X^\star = \lambda_1 v_1 v_1^\tr \), \( Z^\star = \sum_{i=2}^n \lambda_i v_i v_i^\tr \), and \( b = \Amap X^\star \). To generate the cost function matrix \( C \), we generate \( y^\star \in \mathbb{R}^m \) with each entry generated by a uniform distribution between zero and one, and we set \( C = Z^\star + \mathcal{A}^*y^\star  \). Finally, the scalar $a$ is set as some constant such that $\Trace(\Xstar) < a$. 

We note that the solution pair \( (X^\star, y^\star) \) generated above is a pair of optimal solutions as $\Amap \Xstar = b$ and 
\begin{align*}
    \innerproduct{C}{\Xstar} & = \innerproduct{Z^\star + \mathcal{A}^*y^\star}{\Xstar}  \\
    & = \innerproduct{Z^\star}{\Xstar} + \innerproduct{\mathcal{A}^*y^\star}{\Xstar} \\
    & = 0 + \innerproduct{\Ajmap \Xstar}{ y} \\
    & = \innerproduct{b}{\ystar}\\
    &=  \innerproduct{b}{\ystar} + \min_{X \in \snp, \Trace(X) \leq a} \innerproduct{C-\Ajmap \ystar}{X},
\end{align*}
where the fourth equality uses the definition of an adjoint map, i.e., $\innerproduct{\Ajmap y}{ X} = \innerproduct{y}{\Amap X}, \forall y \in \RR^m, X \in \sn$, and the last inequality is due to $C - \Ajmap \ystar $ is positive semidefinite and $a > \Trace(\Xstar) \geq \lambda_1$, which implies $0 = \min_{X \in \snp, \Trace(X) \leq a} \innerproduct{C-\Ajmap \ystar}{X}$. Thus, the generated SDP does contain a rank-one solution $\Xstar$. 

Aside from the description of the problem generation, this subsection also demonstrates that exact linear convergence happens numerically (as the theoretical guarantee in \cref{subsection:linear}) and \alg{} can achieve extremely high accuracy. As shown in \cref{thm:linear-convergence}, linear convergence happens when quadratic growth \eqref{eq:QG} holds for the dual function $g$ and the dual approximated function $g_k$ captures the true function $g$ up to a quadratic term \eqref{eq:quadratic-close}. Those conditions are indeed satisfied in the application of SDP as shown in \cite{ding2023revisiting} when the iterate is close to the solution set. We consider four SDPs with the dimensions ranging from $10$ to $40$. Each SDP is constructed in the same way as the procedure introduced above to admit at least a rank-one solution. We first use the solver Mosek \cite{mosek} to obtain a $10^{-5}$-solution and run \alg{} with that initialization. In \cref{fig:numerical-experiment}, we present the evolution of the residuals (including the primal cost value gap, primal feasibility, and the dual cost value gap). We see that linear convergence does happen. In particular, the dual cost value gap ($|g(y) - g^\star|/|g^\star|$) achieves the accuracy of $10^{-11}$ in all instances, which is more accurate than commercial solvers developed based on the interior point method, such as Mosek \cite{mosek}. Typically, commercial solvers can only achieve the accuracy of $10^{-9}$ due to the numerical error that occurred in the algorithm. Thus, \alg{} might serve as a downstream algorithm to achieve a high-accuracy solution.

\begin{figure}[t]
     \centering
{\includegraphics[width=1\textwidth]
{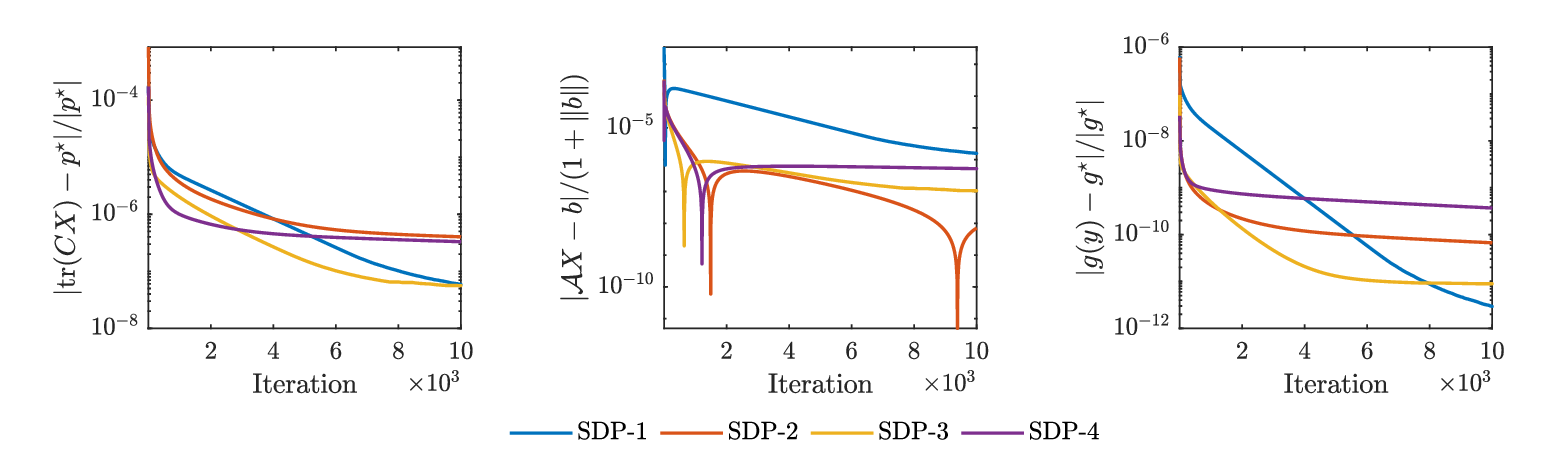}}
\caption{Linear convergence for SDPs.}
\label{fig:numerical-experiment}
\end{figure}

\subsection{Further numerical experiments on practical problems}

\subsubsection{Matrix Completion}
Our second numerical experiment considers the nuclear norm minimization for matrix completion problem \cite{candes2012exact} with the form 
\begin{equation*}
    \begin{aligned}
         \min_{W_1, W_2} & \quad \Trace(W_1) + \Trace(W_2)\\ \mathrm{subject~to} & \quad X_{ij} = (X_{\sharp})_{ij}, \;\forall (i,j) \in \mathcal{G}, \\
         &\quad \begin{bmatrix}
             W_1,X \\
             X^\tr,W_2 
         \end{bmatrix} \in \mathbb{S}_+^{n}.
    \end{aligned}
\end{equation*} 
where $X_{\sharp}$ is the ground truth matrix that we wish to recover, and $\mathcal{G} \subseteq [n] \times [n]$ is an index set of the observed entries. We first generate a vector $w_{\sharp} \in \RR^{n/2}$, and then set $X_{\sharp} =w_{\sharp} w_{\sharp}^\tr$ and let the probability of observing each $X_{\sharp}$ entry as $0.2$, which defines the linear map $\Amap$ and the vector $b$. The cost matrix $C$ is simply $I_{n}$. We set the constant $a$ as some number such that $a > 2\Trace(X_{\sharp})$ as it is known that the matrix $ \begin{bmatrix}
    X_{\sharp} & X_{\sharp}\\
    X_{\sharp} & X_{\sharp}
\end{bmatrix}$
solves the matrix completion problem with high probability \cite[Theorem 5.1]{ding2021simplicity}.

\subsubsection{Binary Quadratic Program}
\label{apx:subsection:BQP}
Our third numerical experiment considers the binary quadratic program (BQP) with the form 
\begin{align*}
 \min_{x \in \RR^n} \quad &x^\tr Q x + q^\tr x \\
\text{subject~to} \quad & x_i^2 = 1, \quad i = 1, \ldots, n,
\end{align*}
where $Q \in \mathbb{S}^n$ and $q \in \RR^n$ are problem data. BQP is a typical example in the field of combinatorial optimization. As BQP is also a polynomial optimization, it can be relaxed as an SDP using the moment relaxation \cite{lasserre2001global}. It is also known that the SDP raised from the moment relaxation admits low-rank solutions. When the relaxation is second-order, empirical evidence shows that the moment relaxation is tight. Each entry of the symmetric matrix $Q$ and the vector $q$ is generated randomly following the uniform distribution with mean $0$ and variance $1$.

\subsubsection{Polynomial optimization over a sphere}
Our fourth numerical experiment considers optimizing a polynomial over a unit sphere 
\begin{align*}
 \min_{x\in \RR^n} \quad & p(x)  \\
\text{subject~to} \quad & x^\tr x = 1,
\end{align*}
where $p:\RR^n \to \RR$ is a polynomial with the degree of four and $n$ variables. Similar to \cref{apx:subsection:BQP}, we consider the second-order moment relaxation of the above polynomial optimization, which is an SDP. As $p$ is a polynomial, it can be written as 
\begin{align*}
    p(x) = \mathbf{c}^\tr [\mathbf{x}]
\end{align*}
where $[\mathbf{x}]$ is the standard monomial of degree four with $n$ variables and $\mathbf{c}$ is the corresponding coefficient vector. We then generate the vector $\mathbf{c}$ randomly with each entry drawn from a normal distribution with mean $0$ and variance $1$.

\begin{figure}[t]
     \centering
     {\includegraphics[width=0.49\textwidth]
{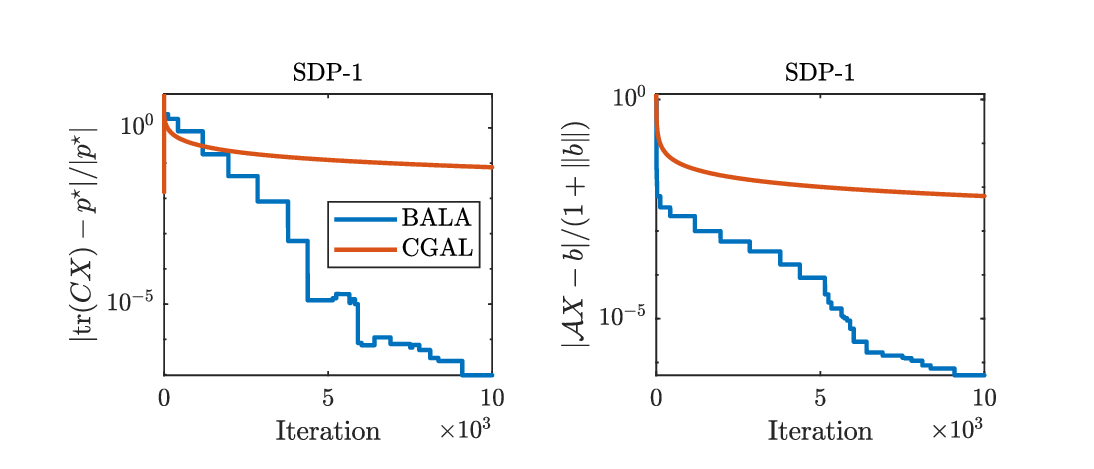}}
{\includegraphics[width=0.49\textwidth]
{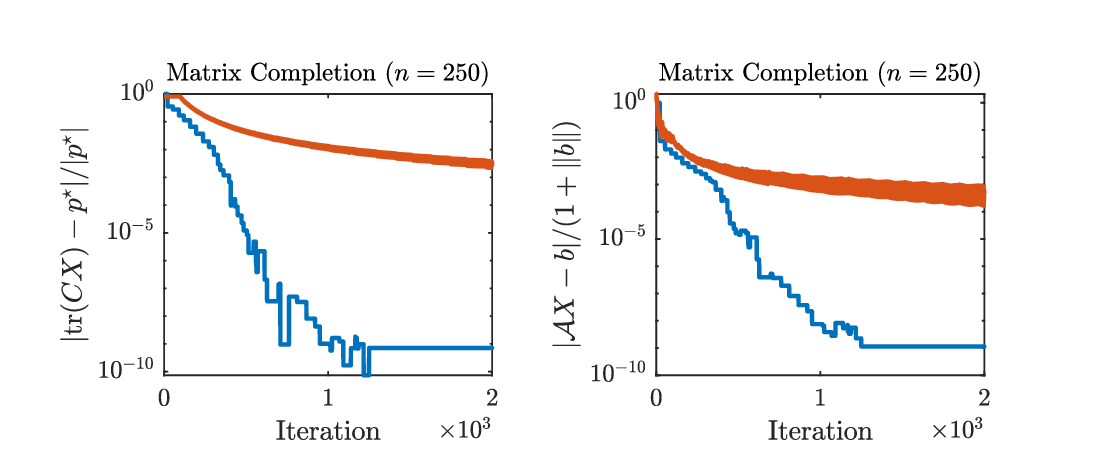}}
{\includegraphics[width=0.49\textwidth]
{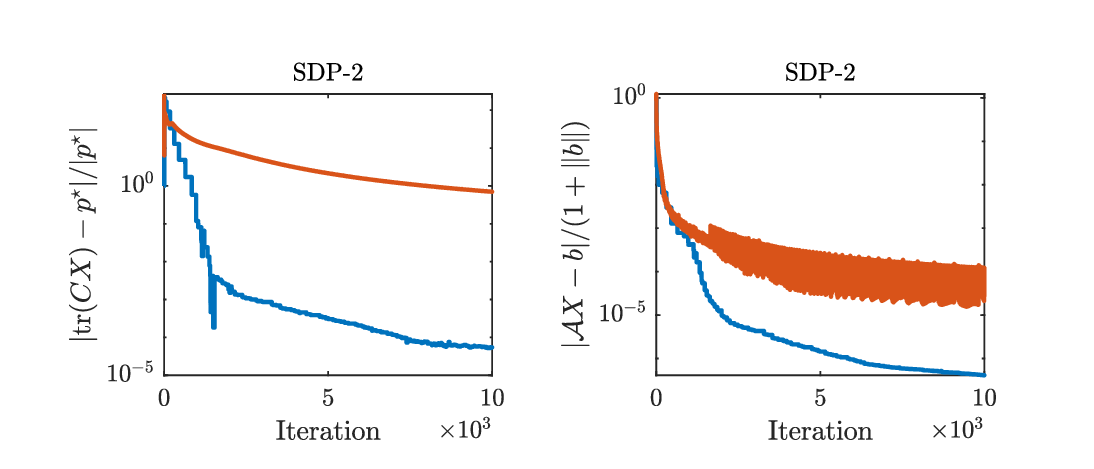}}
{\includegraphics[width=0.49\textwidth]
{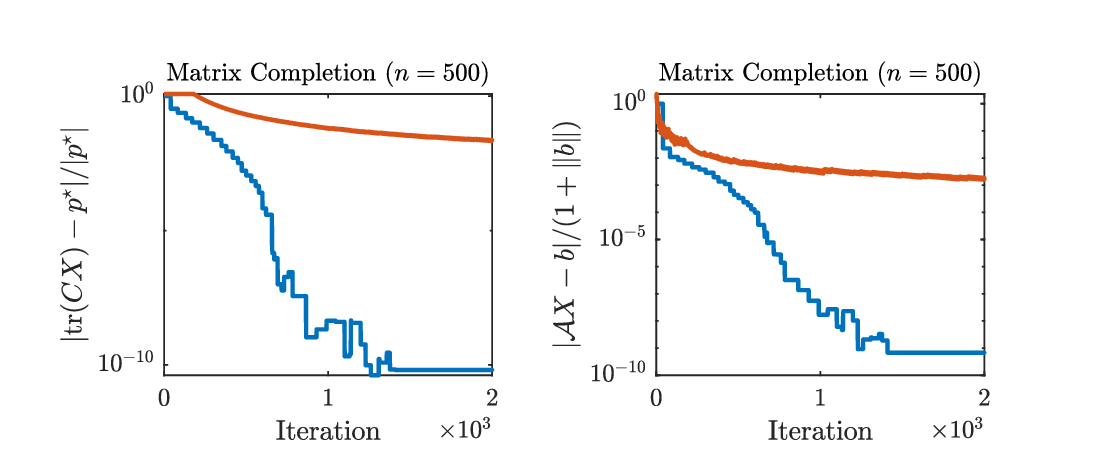}}
{\includegraphics[width=0.49\textwidth]
{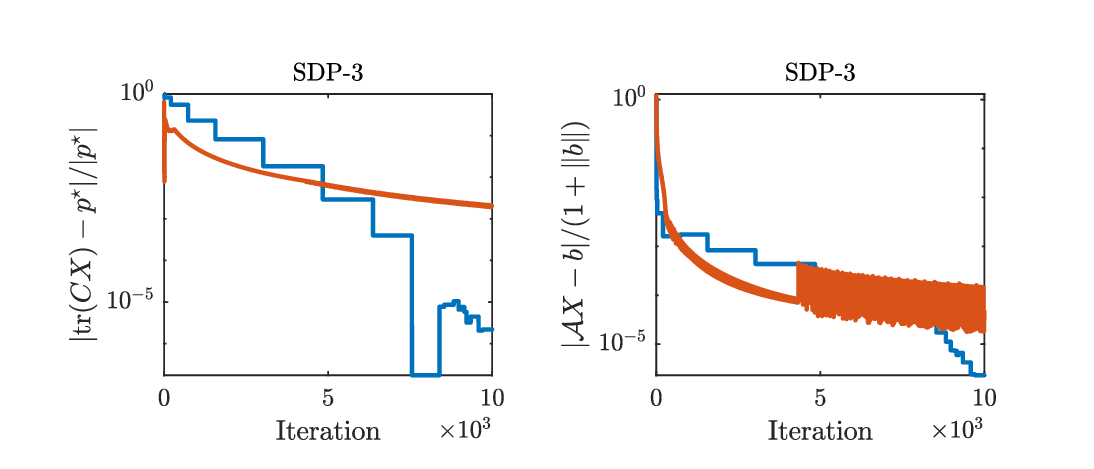}}
{\includegraphics[width=0.49\textwidth]
{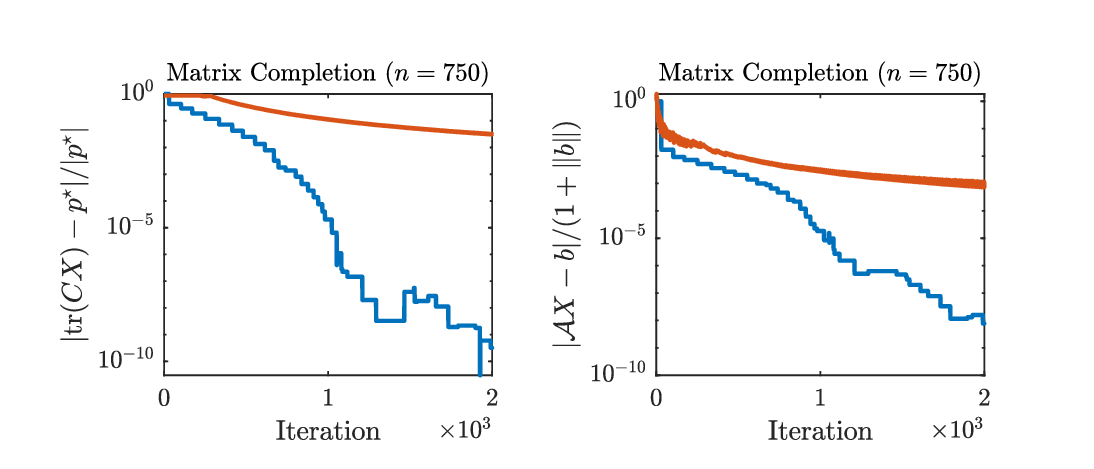}}
     {\includegraphics[width=0.49\textwidth]
{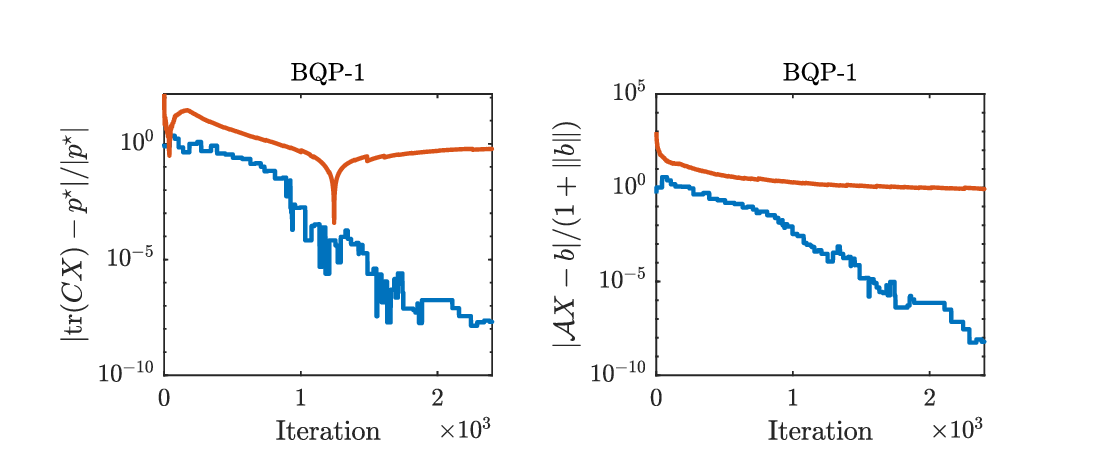}}
{\includegraphics[width=0.49\textwidth]
{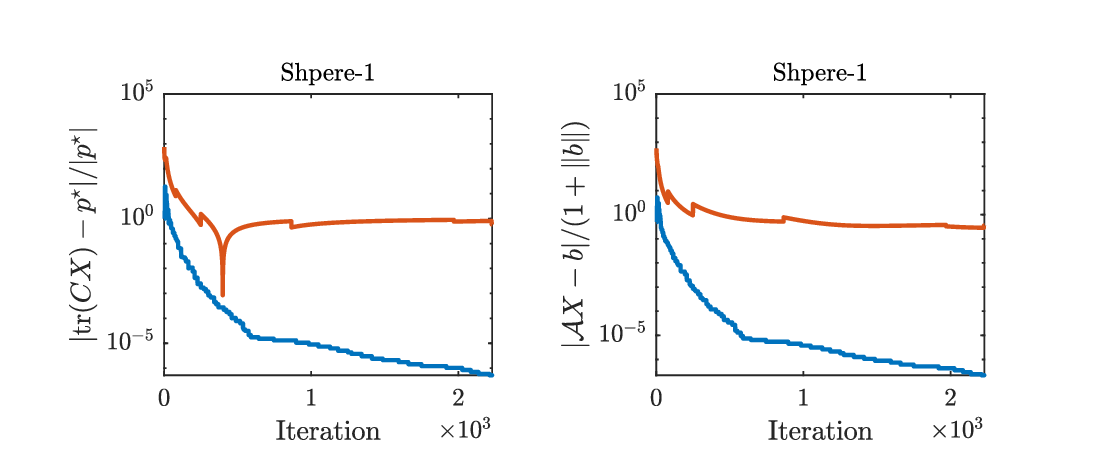}}
{\includegraphics[width=0.49\textwidth]
{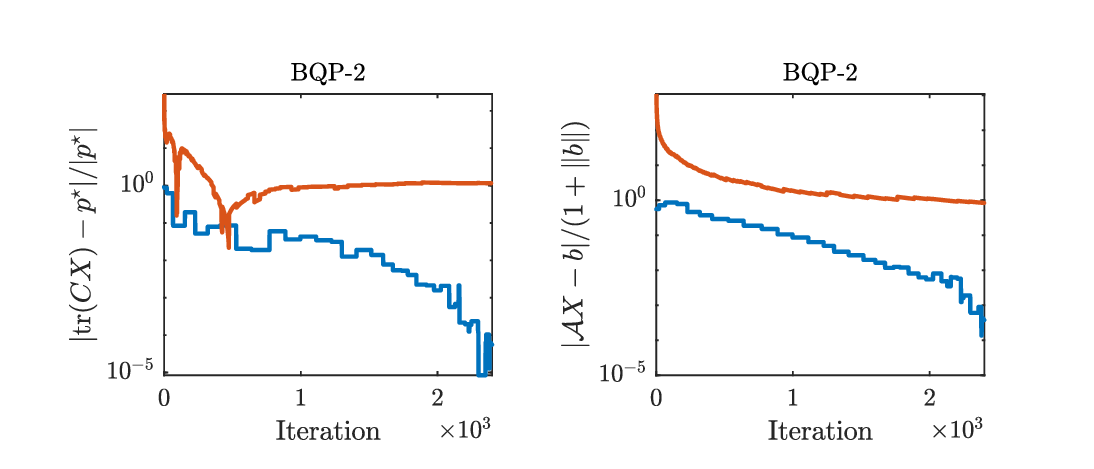}}
{\includegraphics[width=0.49\textwidth]
{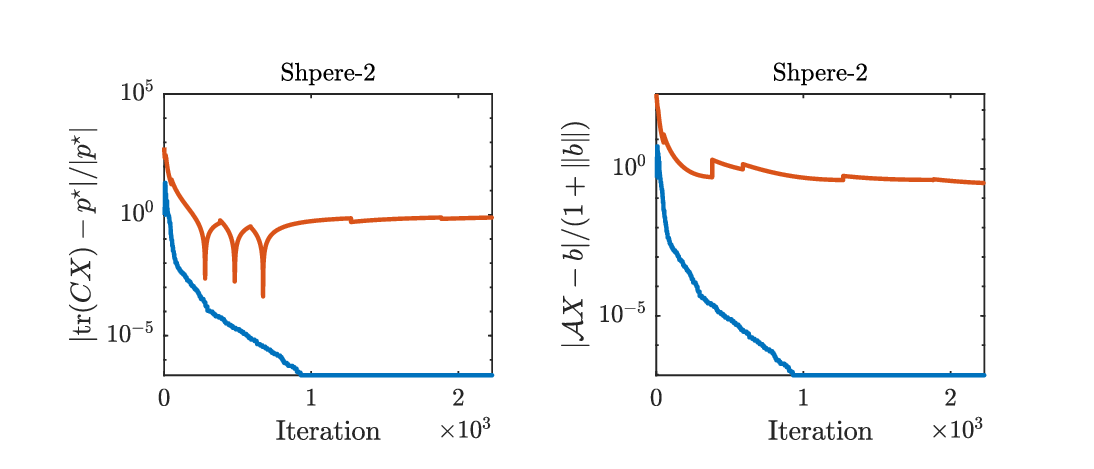}}
{\includegraphics[width=0.49\textwidth]
{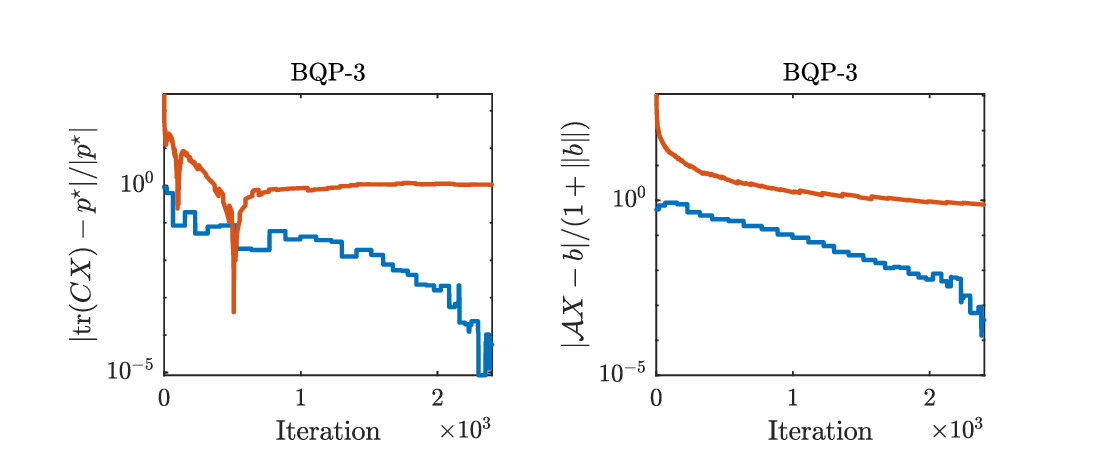}}
{\includegraphics[width=0.49\textwidth]
{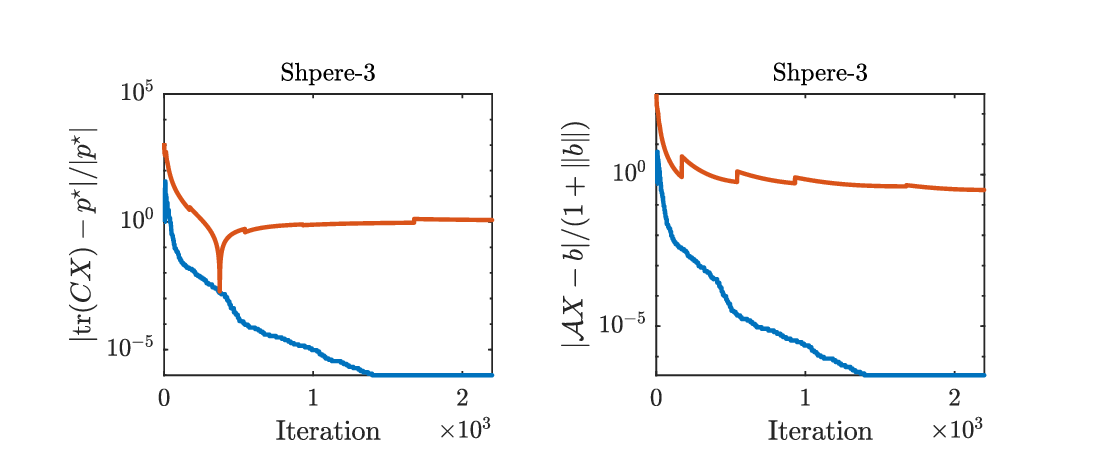}}
\caption{Additional numerical experiments. The blue curves are from \alg{}, and the red curves are from \texttt{CGAL} \cite{yurtsever2021scalable}.}
\label{fig:numerical-experiment-more-1}
\end{figure}

\end{document}